\documentclass[11pt]{amsart}
\usepackage{rlepsf,  
graphicx,
 amssymb, epstopdf, 
amsmath, amssymb,
psfrag, pinlabel, rotating,hyperref
}

\usepackage{fullpage}
\usepackage[active]{srcltx}

\theoremstyle{plain}
\newtheorem{thm}{Theorem}[subsection]
\newtheorem{cor}[thm]{Corollary}

\newtheorem{prop}[thm]{Proposition}
\newtheorem{lemma}[thm]{Lemma}

\newtheorem*{thmC}{Theorem~\ref{thm:bigdepth}}
\newtheorem*{thmD}{Theorem~\ref{thm:tensionone}}

\newtheorem*{thmF}{Proposition~\ref{prop:mintbLOSS}}
\newtheorem*{thmG}{Theorem~\ref{thm:depth1binding}}
\newtheorem*{thmH}{Theorem~\ref{thm:rationalbennquinineq}}
\newtheorem*{thmI}{Theorem~\ref{thm:rationalselflinking}}

\theoremstyle{definition}
\newtheorem{defn}[thm]{Definition}
\newtheorem{remark}[thm]{Remark}

\newtheorem{example}[thm]{Example}

\theoremstyle{definition}
\newtheorem{question}{Question}[section]
\newtheorem{problem}[question]{Problem}

\newcommand{\calN}{\ensuremath{{\mathcal N}}}

\newcommand{\calK}{\ensuremath{{\mathcal K}}}

\newcommand{\comment}[1]{}
\newcommand{\lk}{\ell k}

\newcommand{\bdry}{\ensuremath{\partial}}

\DeclareMathOperator{\nbhd}{\calN}

\newcommand{\N}{\ensuremath{\mathbb{N}}}
\newcommand{\Q}{\ensuremath{\mathbb{Q}}}
\newcommand{\R}{\ensuremath{\mathbb{R}}}

\newcommand{\frL}{\ensuremath{\mathfrak{L}}}

\newcommand{\tw}{\ensuremath{{\mbox{\tt tw}}}}
\newcommand{\tb}{\ensuremath{{\mbox{\tt tb}}}}
\newcommand{\rot}{\ensuremath{{\mbox{\tt rot}}}}
\renewcommand{\sl}{\ensuremath{{\mbox{\tt sl}}}}
\newcommand{\std}{\ensuremath{{\mbox{\small{\tt std}}}}}
\newcommand{\cyl}{\ensuremath{{\mbox{\small{\tt cyl}}}}}

\newcommand{\ort}{\ensuremath{{\mbox{\small{\tt o}}}}}
\newcommand{\OT}{\ensuremath{{\mbox{\tiny{\tt OT}}}}}

\begin{document}
\title{Non-looseness of non-loose knots}
\author{Kenneth L.\ Baker}
\address{
Department of Mathematics,
University of Miami,
Coral Gables, FL 33124-4250, USA}
\email{k.baker@math.miami.edu}

\author{S\.{i}nem Onaran}
\address{
Department of Mathematics, 
Hacettepe University, 
06800 Beytepe-Ankara, Turkey}
\email{sonaran@hacettepe.edu.tr}

\begin{abstract}
A Legendrian or transverse knot in an overtwisted contact $3$--manifold is non-loose if its complement is tight and loose if its complement is overtwisted.  We define three measures of the extent of non-looseness of a non-loose knot and show they are distinct.
\end{abstract}

\maketitle


\section{Introduction}

A contact $3$--manifold $(M,\xi)$ is {\em overtwisted} if it contains an overtwisted disk and {\em tight} otherwise.
A Legendrian or transverse knot $K$ in an overtwisted contact $3$--manifold is called {\em non-loose}  (or {\em exceptional} as in \cite{EF}) if the restriction of the contact manifold to its complement is tight.  If instead the complement is overtwisted, then $K$ is called {\em loose}.   That is, $K$ is non-loose if it intersects every overtwisted disk, while $K$ is loose if it is disjoint from some overtwisted disk.   In this article we develop and examine notions of the extent of non-looseness of Legendrian and transverse knots in overtwisted contact structures.  Throughout, our ambient $3$--manifolds will be closed, compact, connected, and oriented and our contact structures will be co-oriented.

For an unoriented Legendrian knot $L$ in a closed overtwisted contact $3$--manifold $(M,\xi)$ we define three invariants, two geometric and one algebraic.  
\begin{itemize}
\item The {\em depth} of $L$, $d(L)$, is the minimum of $|L\cap D|$ over all overtwisted disks $D$ in $M$.  
\item The {\em tension} of $L$, $t(L)$, is the minimum number of stabilizations required to loosen $K$. 
\item The {\em order} of $L$, $\bar{o}(L)$, is the sum of the orders of the $U$--torsion of the LOSS invariant $\mathfrak{L}$  (defined in \cite{LOSS}) of the two orientations on $L$.  Presently this is only defined when $L$ is null-homologous.

\end{itemize}
 If $L$ is loose, then all three of these are $0$.  If $L$ is non-loose then both its depth and the tension are non-zero by definition, though its order may be $0$.

\begin{thm}\label{thm:mainineq}
If $L$ is a Legendrian knot in an overtwisted contact $3$--manifold, then
\[\bar{o}(L) \leq t(L) \leq d(L)\]
where we only consider $\bar{o}(L)$ if $L$ is null homologous. 
\end{thm}

\begin{proof}
The second inequality is Lemma~\ref{lem:stabboundsdepth}.   When $L$ is also null-homologous, the first inequality is Lemma~\ref{lem:tensionLOSS}.
\end{proof}

Indeed, these three invariants are all distinct.

\begin{thm}\label{thm:tensionlessthandepth}
There exist non-loose Legendrian knots $L$ with \[1=t(L)<d(L).\]
\end{thm}

\begin{proof}[Sketch of Proof]
The following two theorems give a surgery characterization of non-loose Legendrian knots with depth $1$ and a surgery construction of non-loose Legendrian knots with tension $1$.  In particular, if $(+1)$--surgery on  a non-destabilizable Legendrian knot $L$ in $(S^3,\xi_\std)$ satisfying $\tb(L) \leq -2$, $\rot(L)<0$, $-\chi (L) < -(\rot(L)+\tb(L)+2)$ is overtwisted, then the surgery dual is a non-loose Legendrian knot $L^*$ with $1=t(L^*)<d(L^*)$. Etnyre-Honda \cite{etnyrehonda-torus} show there is a Legendrian torus knot satisfying these classical constraints that, according to Lisca-Stipsicz \cite{liscastipsicz}, has an overtwisted $(+1)$--surgery. A more detailed proof is given after the proof of Theorem~\ref{thm:tensionone}.
\end{proof}

\begin{thmC}
Suppose $(+1)$--surgery on a Legendrian knot $L$ with tight complement yields an overtwisted manifold with surgery dual knot $L^*$.  Then $d(L^*)=1$ if and only if $L$ is a stabilization.
\end{thmC}

\begin{thmD}
Suppose $(+1)$--surgery on a Legendrian knot $L$ in $(S^3,\xi_\std)$ yields an overtwisted manifold with surgery dual knot $L^*$.  If  $\tb(L) < -1$, $\rot(L)<0$, and $\tb(L) + \rot(L)+2 < \chi(L)$, then $t(L^*) = 1$.
\end{thmD}

\begin{thm}\label{thm:unknotorder}
A non-loose Legendrian unknot $L$ satisfies \[\bar{o}(L)=0 \mbox{ while } t(L)=d(L)=1.\]
\end{thm}
\begin{proof}
Lemma~\ref{lem:nonlooseunknotdt} shows any non-loose Legendrian unknot has depth $1$ and hence, by Theorem~\ref{thm:mainineq} (or just Lemma~\ref{lem:stabboundsdepth}), also tension $1$.  Corollary~\ref{cor:unknotorder} shows they all have order $0$.
\end{proof}

The proof above relies upon Theorem~\ref{thm:EF}, the characterization of non-loose unknots in $S^3$ of Eliashberg-Fraser \cite{EF}, and Corollary~\ref{cor:nonlooseunknot}, its implication for other manifolds.  The key step to Lemma~\ref{lem:nonlooseunknotdt} is the application of Theorem~\ref{thm:bigdepth} to the surgery diagrams given by Plamenevskaya \cite{plamenevskaya} of these non-loose unknots in $S^3$. Corollary~\ref{cor:unknotorder} is a consequence of the following proposition which exploits the behavior of the LOSS invariant under stabilizations.   Here we state it for just $\bar{o}$, but the actual proposition addresses $\frL$ and related invariants as well.

\begin{thmF}
If a null-homologous knot type $\calK$ has a lower bound on the Thurston-Bennequin numbers of its non-loose Legendrian representatives in a given overtwisted contact structure, then $\bar{o}(L) =0$ for each Legendrian representative $L$.
\end{thmF}

We also consider refinements of the above invariants for oriented Legendrian knot and their analogues for transverse knots.
The binding of an open book is naturally a transverse link in the contact structure supported by the open book. If the open book is a negative Hopf stabilization of another open book, then the binding has depth $1$.
\begin{thmG}
Assume an open book with connected binding is a negative Hopf stabilization.  Then the binding $T$, as the non-loose transverse knot in the overtwisted contact structure the open book supports, has $d(T)=t(T)=1$.
\end{thmG}

Dymara \cite{dymara} and Eliashberg-Fraser \cite{EF} established fundamentals about non-loose knots with a focus on non-loose unknots.   Dymara attributes \'Swi\c{a}towski with a Bennequin type inequality for non-loose Legendrian knots which we recall in Theorem~\ref{thm:bennequinineq}.   In his study of the coarse classification of non-loose Legendrian and transverse knots, Etnyre \cite{etnyreovertwisted} gives the associated Bennequin inequality for non-loose transverse knots, Theorem~\ref{thm:selflinking}.    These two theorems are both for null-homologous knots; in the vein of \cite{BE}, we extend them to rationally null-homologous knots.

\begin{thmH}
For a non-loose rationally null-homologous Legendrian knot $L$ of homological order $r$ with rational Seifert surface $\Sigma$,
\[ -|\tb_\Q(L)| + |\rot_\Q(L)| \leq -\frac{1}{r} \chi(\Sigma). \]
\end{thmH}

\begin{thmI}
For a rationally null-homologous non-loose transverse knot $T$ of homological order $r$ with rational Seifert surface $\Sigma$,
\[  \sl_\Q(T) \leq -\frac{1}{r} \chi(\Sigma).\]
\end{thmI}

\subsection{Outline}
We recall the basic concepts and tools for dealing with contact structures and the Legendrian and transverse knots in them in section~\ref{sec:prelim} while section~\ref{sec:OT} covers the basics of overtwisted manifolds and non-loose knots.  Section~\ref{sec:Leg} develops our invariants and main results for non-loose Legendrian knots; section~\ref{sec:trans} addresses non-loose transverse knots. We conclude with a handful of problems and questions in section~\ref{sec:prob}.

\subsection{Acknowledgements}
We would like to thank Keiko Kawamuro and David Shea Vela-Vick for helpful conversations and Steven Sivek for his input.  KB was partially supported by a grant from the Simons Foundation (\#209184 to Kenneth L.\ Baker).  SO is supported by the Scientific and Technological Research Council of Turkey 
(TUBITAK 3501-\# 112T994 to Sinem Onaran) and SO thanks the Max Planck Institute for Mathematics, Bonn for their hospitality.

\section{Preliminaries}\label{sec:prelim}
\subsection{Contact structures}

A contact structure $\xi$ on a $3$--manifold $M$ is a nowhere integrable $2$--plane field, and the pair $(M,\xi)$ forms a contact $3$-manifold.  Locally $\xi$ is orientation preserving diffeomorphic to $\ker \alpha$ for some $1$--form $\alpha$ satisfying $\alpha \wedge d\alpha \neq 0$.  We restrict attention to positive contact structures, those for which $\alpha \wedge d\alpha >0$.   By Darboux's Theorem, e.g.\ \cite{geigesbook}, every point in a (positive) contact manifold $(M,\xi)$ has a neighborhood admitting an orientation preserving diffeomorphism to an open subset of $(\R^3, \xi_{\std})$, the standard contact structure on $\R^3$ where $\xi_{\std} = \ker( dz -y\, dx)$.

\subsection{Legendrian and transverse knots} 
 A particular smooth embedding of an oriented knot $K$ in a contact $3$--manifold $(M,\xi)$ is {\em Legendrian} if its tangent vectors lie in the contact planes:  $T_p K \subset \xi_p$ for every $p \in K$.   
 On the other hand, the knot $K$ is {\em transverse} if its tangent vectors are not in the contact planes: $T_p K \oplus \xi_p \cong T_p M$ for every $p \in K$.  The co-orientation of $\xi$ naturally orients $K$, and hence we always regard transverse knots as oriented knots.

An isotopy through Legendrian embeddings is a {\em Legendrian isotopy}, and an isotopy through transverse embeddings is a {\em transverse isotopy}.
A {\em Legendrian knot} is a Legendrian isotopy equivalence class, and a {\em transverse knot} is a transverse isotopy equivalence class.
A good reference for the fundamentals of Legendrian knots and transverse knots is \cite{etnyrelegandtrans}.

\subsubsection{Classical invariants}
The most basic invariant of a Legendrian or transverse knot is its topological {\em knot type}. 

 The contact structure endows a Legendrian knot $L$ with a natural framing $\lambda_\xi$ called its {\em contact framing}. Given a surface $\Sigma$ embedded in $(M,\xi)$ that contains $L$ (or even just properly embedded in $M-\nbhd(L)$ and radially extended to $L$ in $\nbhd(L)$), the twist number of $\Sigma$ along $L$ relative to $\xi$ measures how $\Sigma$ twists along $L$ relative to $\xi$.   That is,  $\tw_\xi(L,\Sigma)$ is the slope $(p/q)$ of the curve $\sigma = \Sigma \cap \bdry \nbhd(L)$ where for some orientation $[\sigma] = p[\mu] + q [\lambda_\xi]$.   Here we view the meridian $\mu$ of $L$ and the framing $\lambda_\xi$ in the boundary of a regular neighborhood of $L$, $\bdry \nbhd(L)$ and orient them so that if $L$ is oriented to be parallel to $\lambda_\xi$ then $L$ links $\mu$ once positively.  

{\sc Caution}: When $\Sigma$ is an orientable surface containing $L$, it is common to measure the twisting of $\xi$ along $L$ relative to $\Sigma$ instead (as in the definition of the Thurston-Bennequin number below).   This means our twist number has its sign opposite from what may be more traditional.

 The {\em Thurston-Bennequin number} $\tb(L)$ of a null-homologous Legendrian knot $L$ is the discrepancy between this contact framing and the framing induced by its Seifert surfaces; if $\Sigma$ is a Seifert surface for $L$ then $\tb(L) = -\tw_\xi(L,\Sigma)$.  The {\em rotation number} $\rot(L)$ of an oriented null-homologous Legendrian knot is the winding number of $TL$ after trivializing the contact structure along a Seifert surface.  The {\em self-linking number} $\sl(T)$ of a null-homologous transverse knot $T$ is the linking number of $T$ with a push-off of $T$ in the direction of a nowhere-zero section over $T$ after trivializing the contact structure along a Seifert surface.  These are the ``classical'' invariants of Legendrian and transverse knots (for further details, see \cite{etnyrelegandtrans} for example).
These invariants have been generalized to rational versions for rationally null-homologous knots \cite{BE,ozturk} and to relative versions for any knot in relation to a chosen homologous knot \cite{chernov}. We will use the rational versions $\tb_\Q$ and $\rot_\Q$ in this article.

\subsubsection{Stabilizations}
Let $L$ be an oriented Legendrian knot in a contact $3$-manifold.  By Darboux's Theorem, e.g.\ \cite{geigesbook}, for each point on $L$ there is a neighborhood $N$ with contactomorphism to $(\R^3, \xi_{\std})$
  sending $L\cap N$ to the $x$--axis. In this manner we locally represent $L$ with the front diagram (the projection to the $xz$--plane)
 as in the left-hand side of Figure~\ref{fig:legstab}.  The modification of $L$ to another Legendrian knot $L_+$ as shown in the top right-hand side of Figure~\ref{fig:legstab} is called a {\em positive stabilization} of $L$.  Similarly, $L_-$ shown in the lower right-hand side of Figure~\ref{fig:legstab} is a {\em negative stabilization} of $L$.  We will also write $L_{+a,-b}$ to indicate the Legendrian knot $L$ with $a$ positive stabilizations and $b$ negative stabilizations, for $a,b \geq 0$.  Note it does not matter the order in which these stabilizations are done and it does not matter where along $L$ these stabilizations are done; the results are Legendrian isotopic. However, as demonstrated by the contactomorphism of $\xi_{\std}$ resulting from the $\pi$ rotation about the $z$--axis reversing the image of $L$, whether a stabilization is positive or negative depends on the orientation of $L$: $(-L)_\pm = -(L_\mp)$.

\begin{figure}
\footnotesize
\centering
\psfrag{L}[][]{$L$}
\psfrag{P}[][]{$L_+$}
\psfrag{N}[][]{$L_-$}
\includegraphics[height=1.5in]{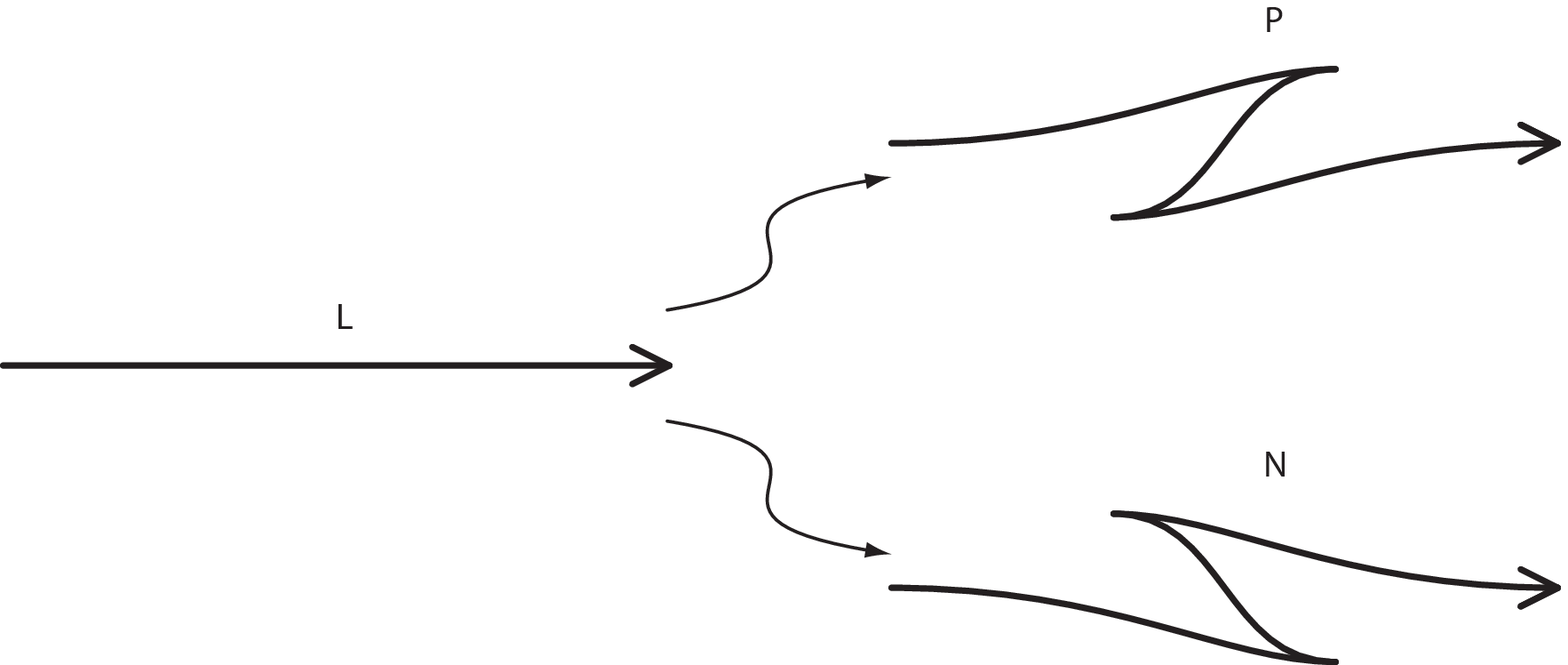}
\caption{}
\label{fig:legstab}
\end{figure}

The effect of stabilizations on the classical invariants for Legendrian knots are 
\[\tb(L_\pm) = \tb(L)-1 \quad \mbox{ and } \quad \rot(L_\pm) = \rot(L)\pm1.\]  
Note that $\tb$ is actually an invariant of unoriented Legendrian knots while $\rot$ requires an orientation.   Indeed, if $-L$ denotes $L$ with the opposite orientation, then $\tb(-L) = \tb(L)$ and $\rot(-L) = -\rot(L)$.    (Let us also note that $\tw_\xi(L_\pm,\Sigma) = \tw_\xi(L,\Sigma)+1$.)

\subsubsection{Transverse push-offs and Legendrian approximations}\label{sec:transverse}
An oriented Legendrian knot $L$, has a positive and negative {\em transverse push-off}, $T_+(L)$ and $T_-(L)$ defined as follows.    Choose a vector field $X$ on $\xi \vert_L$ such that at each point $p \in L$ we have the positive oriented basis $(X(p), T_p(L))$ for $\xi_p$.   Then a small push-off of $L$ in the direction of $X$ gives the transverse knot $T_+(L)$ which is oriented parallel to $L$ by the co-orientation of $\xi$.   A small push-off of $L$ in the direction of $-X$ gives the transverse knot $T_-(L)$ which is oriented parallel to $-L$.   We say the positive transverse push-off $T_+(L)$ is {\em the} transverse push-off.   Negative stabilizations of Legendrian knots do not change the transverse push-off, so any invariant of Legendrian knots that is not altered by negative stabilizations gives an invariant of transverse knots.

Note that, by definition, reversing the orientation of a transverse knot does not produce another transverse knot since its orientation no longer agrees with the co-orientation of the contact structure. 

The contact manifold $(\R^3, \xi_\cyl)$ with $\xi_\cyl = \ker(dz + r^2 \, d\theta)$ is the cylindrical model of the standard contact structure on $\R^3$.  There is a contactomorphism $(\R^3, \xi_\cyl) \to (\R^3, \xi_\std)$.  A regular $\epsilon$--neighborhood of the $z$--axis in this cylindrical model, modulo $z \mapsto z+1$, gives a standard solid torus neighborhood for a transverse knot $T$.  A Legendrian curve $L$ on a concentric torus that is topologically isotopic to the core $T$ of this solid torus is a {\em Legendrian approximation} of $T$.     While there are many different Legendrian approximations to $T$, any two have common negative stabilizations.  Furthermore, the (positive) transverse push-off of a Legendrian approximation of $T$ is again $T$.  See \cite{EFM} for further details.

\subsection{Convex surfaces}

An embedded surface $\Sigma$ in a contact manifold $(M,\xi)$ is  {\em convex} if there is an embedded product neighborhood $\Sigma \times (-1, 1)$ with $\Sigma =  \Sigma \times \{0\}$ such that $\xi$ is preserved by flow in the product direction within this product.  This notion is due to Giroux \cite{girouxconvex}.  He also shows that any surface $\Sigma$ embedded in a contact manifold so that each component $L$ of $\bdry \Sigma$ is Legendrian with $\tw_\xi(L,\Sigma) \geq 0$ admits a $C^0$--isotopy fixing its boundary to a convex representative.  Such an isotopy may be found in any neighborhood of the surface.

Given an oriented convex surface $\Sigma$ with (possibly empty) Legendrian boundary, let $X$ be a vector field given by a vertical flow in the product direction of a product neighborhood of $\Sigma$ (for example $X =\partial_t$ where $t$ is the interval parameter).  
For each point $x \in \Sigma$ the vector $X(x) \in T_x M$ projects to a positive, zero, or negative multiple of the co-orientation of $\xi_x$. The {\em dividing set} $\Gamma$ of $\Sigma$ is the properly embedded $1$--manifold consisting of points for which this is the zero multiple, that is $\Gamma = \{ x \in \Sigma \colon X(x) \in \xi_x\}$.  Then $\Sigma_+$ is the component of $\Sigma - \Gamma$ where the multiple is positive, and $\Sigma_-$ is the component where the multiple is negative.   If $\Gamma'$ is smoothly isotopic to $\Gamma$ in $\Sigma$, there is a smooth isotopy of $\Sigma$, fixing its boundary, through convex surfaces realizing $\Gamma'$ as a dividing set.

If $\Sigma$ is a convex surface with dividing set $\Gamma$ and $L$ is a Legendrian knot in $\Sigma$ (such as component of $\bdry \Sigma$) then $L$ is necessarily transverse to $\Gamma$ and $\tw_\xi(L,\Sigma) = \tfrac{1}{2} |L \cap \Gamma|$.

\subsection{Contact surgery}
A Legendrian knot $L$ has a standard tight neighborhood $\nbhd(L)$ with convex boundary having two dividing curves.  Contact surgery of slope $(p/q)$ is a Dehn surgery on $L$ producing a new contact manifold by replacing $\nbhd(L)$ with another tight contact solid torus having the same boundary but with a meridian of slope $(p/q)$ with respect to the contact framing.  (If $L$ is null-homologous, a $(p/q)$--contact surgery is topologically a Dehn surgery of slope $p/q+\tb(L)$.) While generically there are multiple contact solid tori with the required boundary data, this Dehn surgery is unique when $p = \pm1$.  See \cite{geigesbook} for a more detailed discussion.  In this article we primarily concern ourselves with $(\pm1)$--surgeries.  

After $(1/n)$--surgery on $L$, the core curve of the attached solid torus is again a Legendrian knot $L^*$ called the {\em surgery dual}.  For such surgeries, both $L$ and $L^*$ are isotopic through their solid tori to a curve $L'$ in $\bdry \nbhd(L)$.
We say $L'$ is a {\em Legendrian push-off} of $L$ and the annulus they cobound is a {\em push-off annulus}.  In the surgered manifold, $L'$ may also be viewed as a Legendrian push-off of $L*$.  Through this push-off, an orientation on $L$ confers a natural orientation upon  $L^*$.

\section{Basics on Knots in Overtwisted Contact Structures}\label{sec:OT}

\subsection{Overtwisted disks, loose and non-loose knots}\label{sec:otdisk}
A smoothly embedded disk $D$ with Legendrian boundary such that $\tb(\bdry D)=0$ is an {\em overtwisted disk}.  We orient $D$ so that it induces the orientation on $\bdry D$ with $\rot(\bdry D)>0$. 
 A contact manifold is {\em overtwisted} if it contains an overtwisted disk, and it is {\em tight} otherwise.  

For a convex overtwisted disk, since $\tb(\bdry D) = 0$ it follows that the dividing set $\Gamma$ is disjoint from $\bdry D$ and hence is a collection of simple closed curves in the interior of $D$.

A {\em standard overtwisted disk} is one with neighborhood contactomorphic to that of $D_{\OT} = \{ (r,\theta, 0) \colon r \leq \pi \}$ in the contact manifold $(\R^3, \xi_{\OT})$ where $\xi_\OT = \ker( \cos r \, dz + r \sin r \, d\theta)$.  This disk is convex and has one dividing curve: use the vector field $X = \partial_z$ so that $\Gamma = \{(\pi/2, \theta,0) \}$. Observe that by orienting $D_{\OT}$ so that with the boundary orientation $\rot(\bdry D_{\OT})>0$, the origin is a positive elliptic singularity of the characteristic foliation.

\begin{prop}[Proof of Proposition 4.6.28~\cite{geigesbook}] \label{prop:sot}
Any overtwisted disk $D$ admits a slight isotopy fixing its boundary to an overtwisted disk $D'$ that either is a standard overtwisted disk or properly contains a standard overtwisted disk. \qed
\end{prop}

This gives us an immediate corollary.

\begin{cor}\label{cor:passtosmallOT}
Let $K$ be a Legendrian or transverse knot in an overtwisted contact $3$--manifold.    Then $K$ intersects every overtwisted disk if and only if $K$ intersects every standard overtwisted disk.  \qed
\end{cor}

 As mentioned in the introduction, a Legendrian or transverse knot $K$ in an overtwisted contact $3$--manifold $(M,\xi)$ is {\em non-loose} if $K$ intersects every overtwisted disk (so that its complement is tight) and is {\em loose} if it is disjoint from some overtwisted disk (so that its complement is overtwisted).   Corollary~\ref{cor:passtosmallOT} lets us rephrase this: $K$ is non-loose if it intersects every standard overtwisted disk and loose if it is disjoint from some standard overtwisted disk.   Note that while the complement of any knot in a tight manifold is tight, the term {\em non-loose} is used to imply that the ambient contact manifold is overtwisted.  Furthermore, observe that if some contact surgery on a Legendrian knot is tight then the complement of the knot is tight --- this can be a convenient way to detect non-looseness.

\subsection{Classical invariants and non-loose knots}

If $L$ were a null-homologous Legendrian knot in a tight contact manifold, then the Bennequin Inequality \cite{Benn83,Eli93}  would hold true: $\tb(L) + |\rot(L)| \leq -\chi(L)$.  In an overtwisted contact manifold, one may find Legendrian knots that violate this inequality.  Loose Legendrian knots of arbitrarily large $\tb$ may be constructed by exploiting the overtwisted disks in their complements.   On the other hand, non-loose Legendrian knots could potentially violate the Bennequin Inequality if  $\tb$ were positive and sufficiently large.

\begin{thm}[\'Swi\c{a}towski \cite{dymara,etnyreovertwisted}]\label{thm:bennequinineq}
For a non-loose null-homologous Legendrian knot $L$ with Seifert surface $\Sigma$,
\[-|\tb(L)| + |\rot(L)| \leq -\chi(\Sigma).\]
\end{thm}

\begin{thm}[\cite{etnyreovertwisted}]\label{thm:selflinking}
For a non-loose null-homologous transverse knot $T$ with Seifert surface $\Sigma$,
\[ \sl(T) \leq -\chi(\Sigma). \]
\end{thm}

This follows from the relation
\[ \sl(T_\pm(L)) = \tb(L) \mp \rot(L)\]
where $T_\pm(L)$ denotes the positive and negative transverse push-offs of $L$.

We give an extension of these bounds to rationally null-homologous non-loose knots modeled on the proof of these given in \cite{etnyreovertwisted}.  See \cite{BE} for details on the rational classical invariants and \cite{sinemgeiges} for ways to compute these from contact surgery diagrams.

Recall that a rationally null-homologous knot $K$ in a $3$--manifold $M$ has a {\em rational Seifert surface} (or generalized Seifert surface) which may be obtained from a properly embedded orientable connected surface $\Sigma$ in $M-\nbhd(K)$ such that, when oriented, $\bdry \Sigma$ is a collection of coherently oriented essential curves in the torus $\bdry \nbhd(K)$.  The surface $\Sigma$ may then be radially extended through $\nbhd(K)$ to a singular surface where $K$, the core of $\nbhd(K)$, is the singular set to obtain the rational Seifert surface.  It's convenient however to regard $\Sigma$ as the rational Seifert surface and we say $\chi(\Sigma)$ is the Euler characteristic of the rational Seifert surface.   Observe that the multiplicity $r$ with which $\bdry \Sigma$ covers $K$ is the order of $K$ in homology, $r[K]=0 \in H_1(M)$.

\begin{thm} \label{thm:rationalbennquinineq}
For a non-loose rationally null-homologous Legendrian knot $L$ of homological order $r$ with rational Seifert surface $\Sigma$,
\[ -|\tb_\Q(L)| + |\rot_\Q(L)| \leq -\frac{1}{r} \chi(\Sigma). \]
\end{thm}

\begin{proof}
We follow Etnyre's proof for integrally null-homologous knots in \cite{etnyreovertwisted}. The idea is to take the ``Seifert cable'' $L'$ of $L$ as a set of ruling curves on a convex neighborhood of $L$.  Then we have a null-homologous link $L'$ in the tight exterior of $L$ and may relate $\tb(L')$ to $r\tb_\Q(L)$ and $\rot(L')$ to $r\rot_\Q(L)$.  An application of the Bennequin type inequality to $L'$ will yield the desired result.

Let $\Sigma$ be a rational Seifert surface for $L$, and let $N$ be a standard convex neighborhood of $L$ such that either $L'=\bdry N \cap \Sigma$ is a collection of (Legendrian) ruling curves or a collection of Legendrian divides.  
In the exterior of $L$, which is tight, $L'$ is a null-homologous Legendrian link with Seifert surface $\Sigma' = \Sigma - N$. 

Let $\lambda$ be the homology class of a dividing curve in $\bdry N$ oriented parallel to $L$ and let $\mu$ be the homology class of a meridian in $\bdry N$ linking $L$ positively.   Then $[L'] = r \lambda + r\tb_\Q(L) \mu \in H_1(\bdry N)$.  Assuming that $L'$ is a collection of ruling curves and there are just two dividing curves, the twisting of $L'$ with respect to $\bdry N$ is then $-\frac{1}{2}| 2\lambda \cdot [L']| = -|r\tb_\Q(L)|$.  Because the framings on $L'$ induced by $\bdry N$ and $\bdry \Sigma'$ are equivalent, $\tb(L') = -r|\tb_\Q(L)|$.  If $L'$ is a collection of Legendrian divides, then it follows that $\tb(L')=0=\tb_\Q(L)$.

For the rotation number, trivialize the contact planes in a slightly larger neighborhood of $L$ than $N$ by extending the unit tangent vector of $L$ to a nowhere zero section of $\xi$.  Then any Legendrian curve in $\bdry N$ wrapping $n$ times positively around $L$ has zero winding with respect to this trivialization and thus its rotation number is $n$ times that of $L$.  Hence $\rot(L') = r\rot_\Q(L)$.

\end{proof}

\begin{thm}\label{thm:rationalselflinking}
For a rationally null-homologous non-loose transverse knot $T$ of homological order $r$ with rational Seifert surface $\Sigma$,
\[  \sl_\Q(T) \leq -\frac{1}{r} \chi(\Sigma).\]
\end{thm}

\begin{proof}
Assume $\sl_\Q(T) > -\frac{1}{r} \chi(\Sigma)$.  For a Legendrian approximation $L$ of $T$, Lemma~1.2 \cite{BE} gives $\tb_\Q(L)-\rot_\Q(L) > -\frac{1}{r}\chi(\Sigma)$ since the (positive) transverse push-off of $L$ is $T$.  If $\tb_\Q(L) \leq 0$ then the bound of Theorem~\ref{thm:rationalbennquinineq} fails implying that $L$ is loose.  Hence $T$ is loose by Proposition~\ref{prop:looseLeghasloosepushoff}.   If $\tb_\Q(L) > 0$ then we may perform negative stabilizations of $L$ to create $L'$ with $\tb_\Q(L') \leq 0$.  Since $L'$ is also a Legendrian approximation of $T$, we again conclude that $T$ must be loose.
\end{proof}

\subsection{Non-loose unknots}

A knot that bounds an embedded disk is called an {\em unknot} or a {\em trivial knot}.
\begin{thm}[\cite{EF}]\label{thm:EF}
A Legendrian unknot $L$ in an overtwisted contact manifold is loose if $\tb(L) \leq 0$.
\end{thm}

Eliashberg-Fraser give a {\em coarse} (i.e.\ up to contactomorphism) classification of non-loose Legendrian unknots in overtwisted contact structures on $S^3$ \cite{EF}.  Etnyre and Vogel independently had worked out a similar proof presented in \cite{etnyreovertwisted}.  Geiges-Onaran offer an alternative proof in \cite{sinemgeiges} which extends to a classification of non-loose Legendrian rational unknots in lens spaces.  They explicitly present the classification for $L(p,1)$ and $L(5,2)$, but the technique works for any lens space.

\begin{thm} [\cite{EF}]\label{thm:nonlooseunknotS3}
 On $S^3$, only the overtwisted contact structure $\xi_{-1}$ with Hopf invariant $-1$ contains a non-loose unknot.  For each integer $n>0$ there is a non-loose unknot with classical invariants $(\tb, \rot) = (n, n-1)$ and one with $(\tb, \rot) = (n,-(n-1))$. Up to contactomorphism these are the only.  
\end{thm}

Compare the following with Corollary 2.4 \cite{etnyreovertwisted}.

\begin{cor}\label{cor:nonlooseunknot}
Suppose an overtwisted contact manifold $(M,\xi)$ contains a non-loose Legendrian unknot $L$.  Then there is a non-loose Legendrian unknot $L'$ in $(S^3, \xi_{-1})$ and a tight contact structure $\xi'$ on $M$ such that $(M,\xi) = (M,\xi') \# (S^3,\xi_{-1})$ where $L$ is the image of $L'$.  In particular, a non-loose unknot has classical invariants $(\tb,\rot) = (n, \pm(n-1))$ for some positive integer $n$.
\end{cor}

\begin{proof}
We may assume $M \not \cong S^3$ since the result follows directly from Theorem~\ref{thm:nonlooseunknotS3}.  Thus, since $L$ is an unknot, $M-L$ is reducible and $M-L \cong M \# (S^3-L')$ for an unknot $L'$ in $S^3$.  Since the restriction of $(M,\xi)$ to $M-L$ is tight, by the prime decomposition of tight contact manifolds \cite{colin}, there exists a tight contact structure $\xi^0$ on $M$ and a tight contact structure $\xi'$ on $S^3-L'$ so that  $(M-L, \xi\vert_{M-L}) = (M,\xi^0) \# (S^3-L', \xi')$.  Then $(B^3-L', \xi\vert_{B^3-L'})$, the complement of a neighborhood of a point in $(S^3-L', \xi')$, is contactomorphic to a contact submanifold of $(M-L, \xi\vert_{M-L})$.  As such,  $\xi\vert_{M-L}$ extends across $L$ to $\xi$ so that $L$ is Legendrian and thus it pulls back to an extension of $\xi\vert_{B^3-L'}$ on $B^3-L'$ across $L'$ to a contact structure on $B^3$, and hence to an extension of $\xi'$ on $S^3-L'$ to a contact structure $\xi''$ on $S^3$,  in which $L'$ is Legendrian.  Since $(M,\xi) = (M,\xi^0) \# (S^3, \xi'')$ is overtwisted, $\xi''$ cannot be tight.  Thus $L'$ in $(S^3,\xi'')$ is a non-loose Legendrian unknot.  Theorem~\ref{thm:nonlooseunknotS3} implies $\xi'' =\xi_{-1}$ and describes the possible classical invariants for $L'$. Hence these are the possible classical invariants for $L$ as well.
\end{proof}

Plamenevskaya gives surgery descriptions of the non-loose unknots in $(S^3, \xi_{-1})$ in Figure~3 of \cite{plamenevskaya}.  Corollary~\ref{cor:nonlooseunknot} shows that a non-loose unknot in some other overtwisted manifold may be locally presented by one of these surgery descriptions.

\section{Depth, Tension, and Order for Legendrian Knots}\label{sec:Leg}
We define two geometric invariants and one algebraic invariant of Legendrian knots in a closed overtwisted contact manifold that give measures of non-looseness.  We first address the geometric ones for unoriented knots in section~\ref{sec:dt} and then consider their refinements for oriented knots in section~\ref{sec:ordt}.  Thereafter, in section~\ref{sec:loss} we use the LOSS invariant of \cite{LOSS} to define an algebraic invariant for null-homologous Legendrian knots.

\subsection{Depth and tension} \label{sec:dt}
 
\begin{defn}
The {\em depth}, $d(L)$, of a Legendrian knot $L$ in an overtwisted contact $3$--manifold is the minimum geometric intersection number $|L \cap D|$ taken over all overtwisted disks $D$ transversally intersecting $L$.  The knot $L$ is loose if and only if $d(L) = 0$. 
\end{defn}

\begin{remark}
  By Proposition~\ref{prop:sot}, the depth of a non-loose knot is realized with a standard overtwisted disk.
\end{remark}

\begin{defn}
The {\em tension}, $t(L)$, of a Legendrian knot $L$ in an overtwisted contact $3$--manifold is the minimum total number of stabilizations required to make $L$ loose.   The knot $L$ is loose if and only if $t(L)=0$. 
\end{defn}

The following theorem shows that tension is well-defined. 

\begin{thm}\label{thm:stabtoloose}
There is a finite sequence of stabilizations that loosens a Legendrian knot in a closed, overtwisted manifold.
\end{thm}

\begin{proof}
Let $L$ be a Legendrian knot in the overtwisted manifold $(M,\xi)$ and let $D$ be a standard overtwisted disk transverse to $L$ so that $|L \cap D|$ is a finite set.  If $L$ is disjoint from $D$, then $L$ is already loose.  So assume $|L \cap D|=n>0$.  Since $D$ is a standard overtwisted disk, there exists a contactomorphism $f$ of a neighborhood of $D$ with a neighborhood of $D_\OT$.  Let $\lambda$ be a radial arc in $D_\OT$ from a point of $f(L \cap D)$ to $\bdry D_\OT$ that is disjoint from the origin and the other points of $f(L \cap D)$.  Then there
 is a contactomorphism  $g$ that identifies a neighborhood $(N=\nbhd(\lambda), \xi\vert_N)$ with $(\R^3, \xi_\std)$ such that $g(D_\OT \cap N)$ is the half plane $\{y \geq0, z=0\}$, $g(\lambda)$ is the arc $\{x=0,y\in[0,1], z=0\}$, and $g(f(L) \cap N)$ is the line $\{x=z, y=1\}$.  Then, depending on the orientation of the intersection, the positive or negative stabilization as shown in the front projection of Figure~\ref{fig:legstabtoloosen} reduces $|L \cap D|$ by one. (View the horizontal line as the boundary of the overtwisted disk with the disk going into the page.)  In this manner, repeated stabilizations make $L$ disjoint from $D$ and thus loosen $L$.
\end{proof}

\begin{figure}
\footnotesize
\centering
\psfrag{L}[][]{$L$}
\psfrag{P}[][]{$L_+$}
\psfrag{N}[][]{$L_-$}
\includegraphics[height=1.5in]{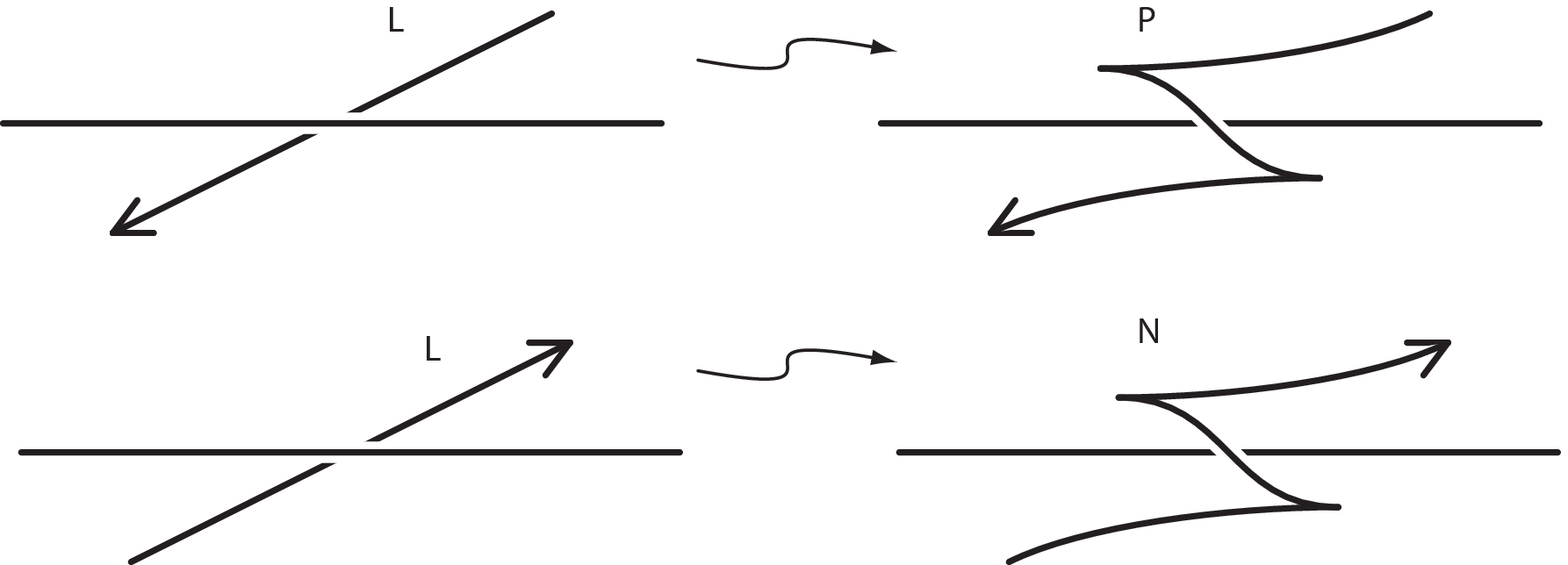}
\caption{}
\label{fig:legstabtoloosen}
\end{figure}

\begin{remark}\label{rmk:stabsign}
Figure~\ref{fig:legstabtoloosen} shows, via the front projection, how stabilizations of oriented Legendrian knots in $(\R^3, \xi_\std)$ may remove intersections with the half plane $\{z=0,y\geq0\}$.   Orienting this half plane so that $\partial/\partial z$ is the positive normal and the boundary orientation is $\partial/\partial x$, one sees that positive stabilizations remove negative intersections while negative stabilizations remove positive intersections.  
\end{remark}

The proof of Theorem~\ref{thm:stabtoloose} offers a relationship between depth and tension.

\begin{lemma}\label{lem:stabboundsdepth}
For a Legendrian knot $L$ in a closed overtwisted contact manifold, $t(L) \leq d(L)$.
\end{lemma}

\begin{proof}
Let $D$ be a standard overtwisted disk for which $|L \cap D|=d(L)$ and then apply the proof of Theorem~\ref{thm:stabtoloose}.  Note that a Legendrian isotopy making $L$ transverse to $D$ may be done without increasing $|L \cap D|$.
\end{proof}

The following two theorems allow us to prove Theorem~\ref{thm:tensionlessthandepth}, that the inequality of Lemma~\ref{lem:stabboundsdepth} is sometimes strict.

\begin{thm}\label{thm:bigdepth}
Suppose $(+1)$--surgery on a Legendrian knot $L$ with tight complement yields an overtwisted manifold with surgery dual knot $L^*$.  Then $d(L^*)=1$ if and only if $L$ is a stabilization.

\end{thm}
\begin{proof}
Lemma~3.1 of \cite{liscastipsicz} shows that $(+1)$--surgery on a destabilizable Legendrian knot in $(S^3,\xi_\std)$  yields an overtwisted manifold. (See also the proof of Theorem~1.2 of \cite{DGS} as well as \cite{ozbagci}.) To prove this they construct an overtwisted disk that is a meridional disk of the surgery solid torus, one that the surgery dual intersects once.  This proof in fact works for $(+1)$--surgery on any destabilizable knot in any contact manifold.  Instead of describing framings in terms of the Thurston-Bennequin number, we'll use the twist number.

Assume $L$ is a stabilization of a Legendrian push-off of a Legendrian knot $L'$ in the contact manifold $(M,\xi)$.  Then there is an annulus $A$ with $\bdry A = L \cup L'$ such that $\tw_\xi(L',A)=0$ and $\tw_\xi(L,A)=+1$.  (Lisca-Stipsicz show that in $(S^3, \xi_\std)$ the annulus $A$ meets $L'$ with framing $\tb(L')$ and $L$ with framing $\tb(L)+1$  \cite{liscastipsicz}.)  Then, in the contact manifold $(M_+,\xi_+)$ obtained by $(+1)$--surgery on $L$, this annulus extends across a meridional disk of a solid torus neighborhood of the surgery solid torus to complete to a disk $D$.  Observe $\bdry D$ is the image of $L'$.  As such, since $\tw_\xi(L',A)=0=\tw_{\xi_+}(\bdry D, D) = \tb(\bdry D)$,  $D$ is an overtwisted disk and the resulting manifold is overtwisted.  By construction, the surgery dual knot $L^*$ intersects $D$ once. 
Thus $d(L^*) =1$ if the complement of $L^*$ is tight.  The complement of $L^*$ is tight if and only if the complement of $L$ is tight.

\smallskip
Now assume $L^*$ is a non-loose knot in an overtwisted manifold $(M_+,\xi_+)$ with $d(L^*)=1$.  Then, by Proposition~\ref{prop:sot}, there is a standard overtwisted disk $D$ which $L^*$ intersects once.  Then, in the contact manifold $(M,\xi)$ obtained by $(-1)$--surgery on $L^*$, $D-L^*$ extends to an annulus $A$ whose boundary is the image of $\bdry D$ and the surgery dual to $L^*$, Legendrian knots $L'$ and $L$ respectively.  Furthermore $\tw_\xi(L',A) = \tw_{\xi_+}(\bdry D,D) = \tb(\bdry D) = 0$ and, due to performing $(-1)$--surgery on $L^*$, $\tw_\xi(L,A)=+1$.  Hence the contact planes of $\xi$ have non-positive twisting along each component of $\bdry A = L \cup L'$ relative to the framing by $A$.   Therefore we may realize $A$ as a convex surface with dividing set $\Gamma$ consisting of a possibly empty set of simple closed curves and a single arc with its endpoints in $L$ .  Since this arc of $\Gamma$ is boundary parallel in $A$, it signifies that $L$ is a stabilization (e.g.\ Lemma 2.20 \cite{etnyrelegandtrans}). 

\end{proof}

\begin{thm}\label{thm:tensionone}
Suppose $(+1)$--surgery on a Legendrian knot $L$ in $(S^3,\xi_\std)$ yields an overtwisted manifold with surgery dual knot $L^*$.  If  $\tb(L) < -1$, $\rot(L)<0$, and $\tb(L) + \rot(L)+2 < \chi(L)$, then $t(L^*) = 1$.

\end{thm}

\begin{proof}
Assume $(M_+,\xi_+)$ is the overtwisted manifold produced by $(+1)$--surgery on $L$.

Let $L'$ be a Legendrian push-off of $L$.  Then the image of $L'$ in $(M_+,\xi_+)$ is a Legendrian push-off of the surgery dual knot $L^*$.  Hence the link $L \cup L'$ with $(+1)$--surgery on $L$ is a surgery diagram for $L^*$ in $(M_+, \xi_+)$.  Similarly, $L \cup L'_+$ with $(+1)$--surgery on $L$ is a surgery diagram for the stabilization $L^*_+$ of $L^*$ in $(M_+,\xi_+)$.

Given that $(M_+,\xi_+)$ is overtwisted by assumption and $L$ is a Legendrian knot in $S^3$, $L^*$ is a non-loose knot.  Hence $t(L^*) >0$.  To show that $t(L^*)=1$ we will demonstrate that the rational classical bound for non-loose knots, Theorem~\ref{thm:rationalbennquinineq}, does not hold for a positive stabilization of $L^*$ if $- \chi (L) < -(\rot(L)+\tb(L)+2)$.   To do so, we will calculate $\tb_\Q(L^*_+)$ and $\rot_\Q(L^*_+)$ using the methods in \cite{sinemgeiges}.  

Observe that $(+1)$--surgery on $L$ is a topological $\tb(L)+1$ -- surgery on $L$;  hence the homological order $r$ of $L^*$ is $|\tb(L)+1|$.  If $\Sigma$ is a rational Seifert surface for $L^*_+$ of minimal genus, then topologically it is the image of a minimal genus Seifert surface for $L$.  Hence $\chi(L) = \chi(\Sigma)$.  Also note that $\lk(L,L'_+)=\lk(L,L') = \tb(L)$ since $L'$ is a Legendrian push-off of $L$ and $L'_+$ is topologically isotopic to $L'$ in the complement of $L$.  Following Lemma~2 of \cite{sinemgeiges}, let $M=(\tb(L)+1)$ be the linking matrix of $L$ and $M_0 = \begin{pmatrix} 0 & \tb(L) \\ \tb(L) & \tb(L)+1 \end{pmatrix}$ be the extended matrix.  Then we may compute:
\begin{align*} 
\tb_\Q(L^*_+) &= \tb(L'_+) + \frac{\det M_0}{\det M} = \tb(L)-1 + \frac{-\tb^2(L)}{\tb(L)+1} = \frac{-1}{\tb(L)+1} \quad \mbox{ and }\\
 \rot_\Q(L^*_+) &= \rot(L'_+) - \langle \rot(L), M^{-1} (\tb(L)) \rangle = \rot(L) + 1 - \langle \rot(L), \frac{1}{\tb(L)+1} \cdot \tb(L) \rangle\\
&=\rot(L) + 1 - \frac{\rot(L) \tb(L)}{\tb(L)+1} = \frac{\rot(L) +(\tb(L)+1)}{\tb(L)+1} 
\end{align*}
where $\tb(L'_+) = \tb(L)-1$ and $\rot(L'_+) = \rot(L)+1$ since $L'_+$ is a positive stabilization of $L$.

  Assume that $L^*_+$, a positive stabilization of $L^*$, is non-loose. Then by Theorem~\ref{thm:rationalbennquinineq} it holds that  
\begin{align*}
-| \tb_\Q(L^*_+) | + |\rot_\Q(L^*_+)| &\leq -\frac{1}{r} \chi(\Sigma) \\
-\left|\frac{-1}{\tb(L)+1}\right| + \left|\frac{\rot(L)+\tb(L)+1}{\tb(L)+1}\right| & \leq -\frac{1}{|\tb(L)+1|} \chi(L) \\
\frac{1}{\tb(L)+1} +\frac{\rot(L)+\tb(L)+1}{\tb(L)+1}  & \leq \frac{\chi(L)}{\tb(L)+1} \\
\rot(L)+\tb(L)+2 & \geq \chi(L)
\end{align*}
since $\rot(L) < 0$ and $\tb(L)+1<0$.
 But this contradicts our assumption that $-(\rot(L)+\tb(L)+2) > -\chi(L)$.  Hence $L^*_+$ is loose and thus $t(L^*) = 1$.
\end{proof}

 In the introduction, we sketched a proof of Theorem~\ref{thm:tensionlessthandepth} which asserts that there exists non-loose Legendrian knots $L$ with $1=t(L) = d(L)$.  We now provide a more detailed proof.  Observe that the knots constructed are only rationally null-homologous.

\begin{proof}[Proof of Theorem~\ref{thm:tensionlessthandepth}]
Assume $L$ is a non-destabilizable Legendrian knot in $(S^3,\xi_\std)$ satisfying   $\tb(L) \leq -2$, $\rot(L)<0$, $-\chi (L) < -(\rot(L)+\tb(L)+2)$, and $(+1)$--surgery on $L$ is overtwisted.  Let $L^*$ be the Legendrian knot dual to the $(+1)$--surgery on $L$.  Since $L$ and $L^*$ have contactomorphic complements, $L^*$ is a non-loose knot.   Because $L$ is non-destabilizable, Theorem~\ref{thm:bigdepth} implies that $d(L^*) \geq 2$.   Due to the assumptions on $\tb(L)$, $\rot(L)$, and $\chi(L)$, Theorem~\ref{thm:tensionone} implies that $t(L)=1$.   Now we simply need to confirm that such Legendrian knots exist in $(S^3,\xi_\std)$.

Let $p$ and $q$ be coprime integers satisfying $-p>q>0$.  By Theorem 4.1 of \cite{etnyrehonda-torus} the maximal Thurston-Bennequin number of Legendrian representatives of the negative torus knot $T_{p,q}$ is $pq$.   By Theorem 4.4 of \cite{etnyrehonda-torus}, among these Legendrian representatives with maximal Thurston-Bennequin number is one with rotation number $q+p$.  Take $L$ to be this torus knot; hence $\tb(L) = pq<0$, $\rot(L) = q+p<0$, and $\chi(L) = |q|-|p||q-1| = q-p+pq$.  Since $0>2(p+1)$, it follows that $-\chi (L) < -(\rot(L)+\tb(L)+2)$.   Finally, $(+1)$--surgery on $L$ is overtwisted by Corollary~1.2 of \cite{liscastipsicz}.
\end{proof}

For non-loose Legendrian unknots, the depth and tension are always equal to $1$.

\begin{lemma}\label{lem:nonlooseunknotdt}
If $L$ is a non-loose Legendrian unknot in an overtwisted manifold, then $t(L)=d(L) =1$.
\end{lemma}

\begin{proof}
 Theorem~\ref{thm:bigdepth}  applies to the surgery descriptions of the non-loose Legendrian unknots in $(S^3,\xi_{-1})$ given in  \cite{plamenevskaya} (see also \cite{sinemgeiges}) to show that these knots all have depth $1$.  (More explicitly, each non-loose Legendrian unknot is presented in a surgery description of $(S^3,\xi_{-1})$ as a Legendrian push-off of a stabilized Legendrian unknot on which $(+1)$--surgery is done; as such, the non-loose Legendrian unknot is Legendrian isotopic to the surgery dual of that component.) 
 It then follows from Corollary~\ref{cor:nonlooseunknot} that the surgery diagrams of \cite{plamenevskaya} may be used with  Theorem~\ref{thm:bigdepth}  to show that a non-loose Legendrian unknot in any overtwisted manifold also has depth $1$.  Lemma~\ref{lem:stabboundsdepth} then implies any non-loose Legendrian unknot also has tension $1$.
\end{proof}

\begin{figure}[h!]
\footnotesize
\centering
\psfrag{K}[][]{$K$}
\psfrag{L1}[][]{$L_1$}
\psfrag{L2}[][]{$L_2$}
\psfrag{L}[][]{$L$}
\psfrag{P}[][]{$P$}
\includegraphics[width=6in]{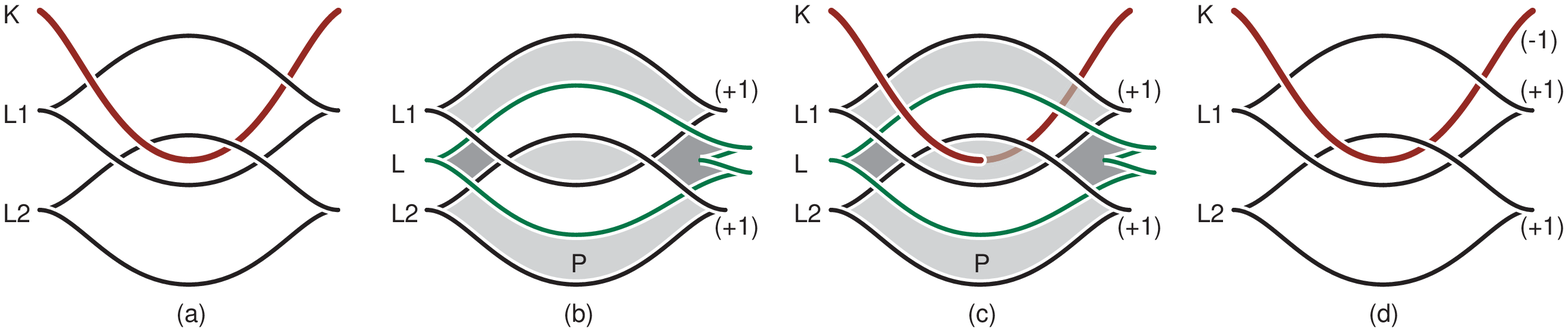}
\caption{}
\label{fig:tb1hopf}
\end{figure}

\begin{prop}\label{prop:ots3knots}
Let $K$ be a Legendrian knot in $(M, \xi)$ with standard Legendrian meridian $L_1$ and its Legendrian push-off $L_2$, locally pictured as in the front diagram of Figure~\ref{fig:tb1hopf}.  
The result of $(+1)$--surgeries on both $L_1$ and $L_2$ is the overtwisted manifold $(M',\xi') = (M,\xi) \# (S^3, \xi_{-1})$.
Letting $K'$ be the image of $K$ in $(M',\xi')$,  then $d(K') \leq 1$.   
Furthermore, if either $(M,\xi)$ or $(+1)$--surgery on $K$ is tight then $d(K')=1$.
\end{prop}
\begin{proof}
Let us first focus locally on the link $L_1 \cup L_2$ of Legendrian unknots of $\tb=-1$ with their $(+1)$--surgeries.  Viewed in $(S^3,\xi_\std)$, this pair of surgeries yields $(S^3, \xi_{-1})$ giving the first claim.  We may construct an overtwisted disk in $(M', \xi')$ as follows:  Figure~\ref{fig:tb1hopf}(b) adds a third Legendrian unknot $L$ of $\tb=-2$ which cobounds a thrice-punctured sphere $P$ with $L_1 \cup L_2$.  Observe that $P$ meets $L_1$ and $L_2$ with topological framing $0$ and $L$ with the contact framing.  Hence, after the pair of $(+1)$--surgeries, $P$ caps off to a disk $D$ where the boundary of $D$ is (the image of) $L$. Since $P$ meets $L$ with the contact framing, $D$ is an overtwisted disk.

Now including $K$ in our consideration, we see that we may choose $L$ and $P$ so that $K$ intersects $P$ once as in Figure~\ref{fig:tb1hopf}(c).  Thus $K'$, the image of $K$ in $(M',\xi')$, intersects the overtwisted disk $D$ once.  Hence $d(K')\leq1$.

If $(-1)$--surgery on $K'$ produces a tight manifold, then $K'$ will be non-loose so that $d(K')=1$.  Locally, a surgery diagram for this manifold is given in Figure~\ref{fig:tb1hopf}(d).  By Proposition~2 of \cite{dinggeiges}, since $L_1$ is a standard Legendrian meridians of $K$, it is Legendrian isotopic to a Legendrian push-off of $K$ after the $(-1)$--surgery on $K$. Since $L_2$ is a Legendrian push-off of $L_1$, it has a Legendrian isotopy following that of $L_1$ as a push-off.  Thus, after $(-1)$--surgery on $K$, $L_1 \cup L_2$ is Legendrian isotopic to a Legendrian push-off of $K$ with a further push-off.  Thus $(+1)$--surgery on $L_1$, say, cancels the $(-1)$--surgery on $K$ returning the manifold $(M, \xi)$ in which $L_2$ is now both the only remaining surgery curve and Legendrian isotopic to the original knot $K$.  Hence the result of $(-1)$--surgery on $K'$ in $(M',\xi')$ may also be obtained by $(+1)$--surgery on $K$ in $(M,\xi)$.

\end{proof}

\begin{remark}
Observe that the surgery duals to $(-1)$--surgery on $K'$ and $(+1)$--surgery on $K$ are Legendrian isotopic knots.
Hence from the conclusion of Proposition~\ref{prop:ots3knots}, Theorem~\ref{thm:bigdepth} implies $d(K')=1$ if and only if the surgery dual to $(+1)$--surgery on $K$ is a stabilization.  
\end{remark}

\begin{lemma}[Proof of Theorem~1.1 \cite{liscastipsicztightI}]\label{lem:possurg}
If $K$ is a Legendrian knot in $(S^3, \xi_\std)$ with $\tb(K) = 2g_s(K)-1 >1$, then $(+1)$--surgery on $K$ is tight. \qed
\end{lemma}
Here $g_s(K)$ denotes the smooth $4$--ball genus of a knot $K$ in $S^3$.

\begin{example}
Let $p,q$ be positive coprime integers.
Let $K$ be a Legendrian $(p,q)$--torus knot in $(S^3, \xi_\std)$ with $\tb(K) =pq-p-q =2g_s(K)-1$. (Such Legendrian knots do exist, e.g.\ \cite{etnyrehonda-torus}.)  Thus by Proposition~\ref{prop:ots3knots} and Lemma~\ref{lem:possurg}, $(+1)$--surgeries on a standard Legendrian meridian and its push-off send $K$ to a Legendrian $(p,q)$--torus knot in $(S^3,\xi_{-1})$ with depth $1$.
\end{example}

\begin{example}
An upper bound on tension may be obtained from a violation of Theorem ~\ref{thm:bennequinineq} after some number of stabilizations.  Here we show that there are non-loose Legendrian knots $K$ with $t(K)=1$ for which Theorem~\ref{thm:bennequinineq} only offers a large upper bound.

Continuing with the previous example,
 Figure ~\ref{fig:nonloosetorusknots} gives an explicit surgery diagram for non-loose Legendrian $(2,q)$--torus knots $L_{2,q}$ in $(S^3, \xi_{-1})$ that have depth and tension equal to $1$.  A straightforward computation with the formula from Lemma 6.6 in \cite{LOSS} shows that the invariants of these non-loose torus knots are $\tb(L_{2,q}) = q$ and $\rot(L_{2,q}) =0$.
So after $q$ stabilizations this knot will violate the inequality in Theorem~\ref{thm:bennequinineq} thereby only implying that $t(L_{2,q}) \leq q$.  
\end{example}

\begin{figure}[h!]
\centering
\includegraphics[height=2in]{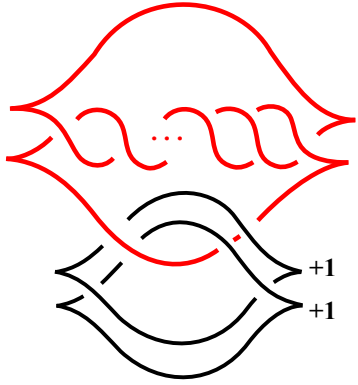}
\caption{Non-loose $(2,q)$-torus knots in $(S^3, \xi_{-1})$.}
\label{fig:nonloosetorusknots}
\end{figure}

Theorem~\ref{thm:bigdepth} may be generalized to give characterizations of Legendrian knots of larger depth.  Topologically, if a knot $K^*$ transversally intersects the interior of a disk $D$ a total of $n$ times, then each of these intersections may be tubed to a meridian $m^*$ of $K^*$ creating a ``folded surface''  $\Sigma$ where the ``boundary'' of $\Sigma$ is $\bdry D$ together with $n$ copies of $m^*$.  (The terminology ``folded'' comes from Turaev \cite{turaev}.  When the algebraic intersection number of $K^*$ with $D$ is $\pm n$, then this folded surface is a rational Seifert surface for the link $\bdry D \cup m^*$.)  Being disjoint from $K^*$, this folded surface $\Sigma$ persists through surgery on $K^*$.  Moreover, after integral surgery on $K^*$, the surgery dual knot $K$ is isotopic to $m^*$ and hence $\Sigma$ is a folded surface for the link $\bdry D \cup K$ with multiplicity $1$ on $\bdry D$ and multiplicity $n$ on $K$.

In Theorem~\ref{thm:depth2} below, we examine the above construction in the contact setting for $n=2$ where the folded surface $\Sigma$ is an honest surface, a once-punctured torus or Klein bottle.

\begin{thm}\label{thm:depth2}
Suppose $(+1)$--surgery on a Legendrian knot $L$ in $(M,\xi)$ with tight complement yields an overtwisted  manifold with surgery dual knot $L^*$.   Then $d(L^*) = 2$ if and only if (i) $L$ is not a stabilization and (ii) there is a once-punctured torus or once-punctured Klein bottle $\Sigma$ with $\bdry \Sigma$ Legendrian and $\tw_\xi(\bdry \Sigma, \Sigma)=0$ that contains $L$ as an essential, non-separating, orientation preserving curve in $\Sigma$ with $\tw_\xi(L, \Sigma) = +1$. 
\end{thm}

Note that while every isotopy class of essential, simple closed curves in a once-punctured torus
is non-separating and orientation preserving, there is a unique such class in a once-punctured Klein bottle.

\begin{proof}
This proceeds in the same manner as the proof Theorem~\ref{thm:bigdepth}. 

Assume $L$ is a curve in a surface $\Sigma$ as in part (ii) of the hypotheses so that $\Sigma - L$ is a thrice-punctured sphere. Then $(+1)$--contact surgery on $L$ topologically compresses $\Sigma$ to a disk $D$ that $L^*$ intersects twice --- that is, $(+1)$--contact surgery fills the exterior of $L$ to cap-off  two boundary components of $\Sigma - \nbhd(L)$.  Since $\bdry D  = \bdry \Sigma$, $D$ is an overtwisted disk.   Hence $d(L^*) \leq 2$.  Because $L$ is not a stabilization by (i) but has tight complement, $d(L^*) \geq 2$ by Theorem~\ref{thm:bigdepth}.  Thus $d(L^*) = 2$.

Now $L*$ is a Legendrian knot in the overtwisted manifold $(M_+,\xi_+)$ and assume $d(L^*) = 2$.  Let $D$ be a standard overtwisted disk that $L^*$ intersects twice transversally. Take a $\tb=-1$ Legendrian meridian $m^*$ of $L^*$ so that $L'$ is isotopic to the Legendrian knot dual to $(-1)$--surgery on $L^*$.
Then tube $D$ along an arc of $L^*$ through $m^*$ to form $\Sigma$, a once-punctured torus or once-punctured Klein bottle (depending whether or not $L^*$ has trivial algebraic intersection number with $D$).  Observe that $\tw_{\xi_+}(m^*,\Sigma) = +1$ by construction.  If $L$ is the surgery dual knot and $(M,\xi)$ is the resulting contact manifold, then since $\xi$ and $\xi_+$ agree on the complements of neighborhoods of $L$ and $L^*$, we may Legendrian isotop $L$ into the position of $m^*$ so that $\tw_\xi(L,\Sigma) = +1$.  By construction $\tw_\xi(\bdry \Sigma, \Sigma)=0$ and $L$ is not a stabilization by Theorem~\ref{thm:bigdepth}.
\end{proof}

\begin{remark}\label{rem:depth2constructions}
As of this writing, we have yet to use Theorem~\ref{thm:depth2} to construct any explicit examples of Legendrian knots $L^*$ of depth $2$ arising as surgery duals to knots in $(S^3,\xi_\std)$.  Assumably such knots exist.   The knots we were able to construct satisfying (ii) of Theorem~\ref{thm:depth2} ended up being stabilizations.  

One approach is to begin with a genus $1$ Legendrian knot $K$ in a tight $(M,\xi)$ with $\tb(K)=0$.  Then $K$ has a once-punctured torus Seifert surface $\Sigma$ which may be made convex with dividing set $\Gamma$ disjoint from $K = \bdry \Sigma$.   Assuming $\Gamma$ has just two (parallel) components, let $L$ be an essential curve in $\Sigma$ that intersects each component of $\Gamma$ once.  We may assume $L$ is Legendrian by the Legendrian Realization Principle \cite{kanda,honda1} and thus conclude that $\tw_\xi(L, \Sigma) = +1$.  The only thing left to check is whether or not $L$ is a stabilized knot.

Another approach is to begin with a non-stabilized Legendrian knot $L$ in a tight $(M,\xi)$ (or just with tight complement) for which $(+1)$--contact surgery is overtwisted and attempt to construct the surface $\Sigma$ of Theorem~\ref{thm:depth2}.    Take a pair of Legendrian push-offs and insert an extra twist to form an annulus $A$ that contains $L$ with $\bdry A$ Legendrian and each of these three curves having $\tw_\xi(\bullet, A) = +1$.  Then the trick is to find a Legendrian banding of the components of $\bdry A$ that results in a surface $\Sigma$ with $\tw_\xi(\bdry \Sigma, \Sigma) = 0$.
\end{remark}

\subsection{Orientation refinements}\label{sec:ordt}

As defined, depth and tension do not depend upon the orientation of a Legendrian knot.   Nevertheless, if a non-loose knot is oriented, one may care to keep track of the oriented intersections with an overtwisted disk and the signs of stabilizations that loosen it.  Recall from section~\ref{sec:otdisk} that an overtwisted disk $D$ is oriented so that $\rot(\bdry D)>0$ and from Remark~\ref{rmk:stabsign} that a positive stabilization removes a negative intersection of an oriented Legendrian knot with a standard overtwisted disk while a negative stabilization removes a positive intersection.
 
\begin{defn}
Let $L$ be an oriented Legendrian knot in an overtwisted contact structure.

Define $d_\ort(L)$, the {\em oriented depth} of $L$, to be set of pairs $(p, n)$ such that $L$ transversally intersects some overtwisted disk $D$ positively $p$ times and negatively $n$ times with $p+n = d(L)$.

Define $t_\ort(L)$, the {\em oriented tension} of $L$, to be the set of pairs $(a, b)$ such that the stabilization $L_{+a,-b}$ is loose and $a+b = t(L)$. 
  
\end{defn}

\begin{defn}
Let $L$ be an oriented Legendrian knot in an overtwisted contact structure.  Define $t_+(L)$,  the {\em positive tension} of $L$,  to be the minimum number of positive stabilizations required to loosen $L$ if possible, and $\infty$ otherwise.  Similarly define $t_-(L)$, the {\em negative tension} of $L$. Observe that $t_\pm(-L)=t_\mp(L)$ where $-L$ is $L$ with its orientation reversed.

\end{defn}

Due to their relationship to the LOSS invariant and relevance for transverse knots, we choose to focus upon positive and negative tensions instead of the oriented depth and tension.

\begin{figure}
\footnotesize
\centering
\psfrag{P}[][]{$L^0$}
\psfrag{L}[][]{$L$}
\psfrag{D}[][]{$L'$}
\psfrag{i}[][]{{\tiny isotopy}}
\psfrag{s}[][]{\tiny neg.\ stab.}
\includegraphics[width=\textwidth]{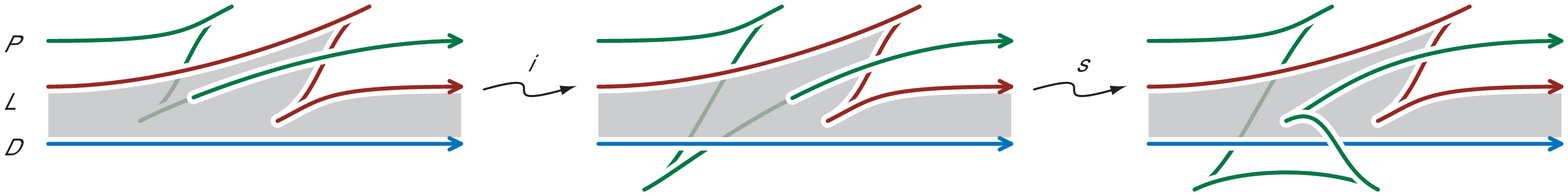}
\caption{}
\label{fig:annulusstab}
\end{figure}

We may partially extend Theorem~\ref{thm:bigdepth} for these signed tensions.
\begin{prop}\label{prop:tensionrefinement} 
Assume the oriented Legendrian knot $L$ in the contact manifold $(M,\xi)$ is a positive stabilization of the Legendrian knot $L'$ and has tight complement. Let $L^*$ be the oriented knot surgery dual to $(+1)$--surgery on $L$.
 
Then $t_-(L^*) = 1$.  
Furthermore, if the Heegaard Floer contact invariant for $(+1)$--surgery on $L'$ in $(M,\xi)$ is non-zero, then $t_+(L^*)>1$. 
\end{prop}

Ozsv\'ath-Szab\'o's Heegaard Floer contact invariant is defined in \cite{osctct}; its behavior under connected sums and $(-1)$--contact surgery follows from this. 

\begin{proof}
Let $L$ be a positive stabilization of a Legendrian push-off of a Legendrian knot $L'$, and let $A$ be the annulus between them. 
Let $L^0$ be a further Legendrian push-off of $L$.    This situation is illustrated in Figure~\ref{fig:annulusstab} on the left.   It further indicates how a negative stabilization of $L^0$ can remove the intersection with $A$.

 As in Theorem~\ref{thm:bigdepth} and Lemma~3.1 of \cite{liscastipsicz}, after $(+1)$--surgery on $L$, $L^0$ is Legendrian isotopic to the surgery dual knot $L^*$ and $A$ caps off to an overtwisted disk $D$ with $\bdry D = L'$.  Thus a negative stabilization on $L^*$ can remove the intersection with $D$, and so $t_-(L^*)=1$.

\medskip
 We will now show that $(-1)$--surgery on a positive stabilization $L^*_+$ of $L^*$ has non-zero contact invariant if $(M,\xi)$ does too in order to conclude that $L^*_+$ is non-loose and hence $t_+(L^*)>1$.  To do so, we will first construct an open book supporting the contact manifold obtained by $(-1)$--surgery on $L^*_+$ and relate it to open books supporting $(M,\xi)$ and $(+1)$--surgery on $L'$.

Let $(\Sigma, \phi)$ be an (abstract) open book supporting $(M,\xi)$ in which the page $\Sigma$ contains the Legendrian knot $L'$ as a non-separating curve so that the page framing agrees with the contact framing.  Giroux shows how to do this \cite{girouxopenbook}.  We may assume $\phi$ acts trivially in a collar neighborhood of $\bdry \Sigma$.   Let $a$ be an embedded arc in $\Sigma$ with one endpoint on $\bdry \Sigma$, the other endpoint on $L'$, and interior disjoint from $L'$.  A neighborhood of $a$ is shown in Figure~\ref{fig:abstrOB}(a).  Let $A$ be a regular annular neighborhood of $L' \cup a$.  We will be modifying $(\Sigma,\phi)$ within this neighborhood.

A single stabilization of this open book near $a \cap \bdry \Sigma$ allows the Legendrian knot $L$ (a positive stabilization of $L'$) to embed in this stabilized page with page framing equal to its contact framing.  This is shown in Figure~\ref{fig:abstrOB}(b) along with a parallel copy $L^0$ of $L$.  A further stabilization of the open book then lets $L^0_+$, a positive stabilization of $L^0$, embed in the page with page framing equal to its contact framing as shown in Figure~\ref{fig:abstrOB}(c).  Figure~\ref{fig:lantern}(a) shows this again in the entirety of $A$ with the two stabilizations.
 Let $(\Sigma'', \bdry_2 \circ \bdry_1 \circ \phi'')$ denote this resulting twice stabilized open book for $(M,\xi)$.  Since the stabilizations occur in neighborhood of a point in $\bdry \Sigma$, we may view $\Sigma''$ as a subsurface of $\Sigma$ whose complement is two open disks in a neighborhood of $a$ (as opposed to viewing $\Sigma''$ as the result of attaching two $1$--handles to $\Sigma$) and $\phi''$ is the restriction of $\phi$ to $\Sigma''$.   The maps $\bdry_1$ and $\bdry_2$ indicates positive Dehn twist along a curve parallel to each of the two new boundary components.  

Let $\phi^* =  \tau^{-1}_{L} \circ \bdry_2 \circ \bdry_1$ and $\phi^*_+ = \tau_{L^0_+} \circ \phi^*$.
It now follows that $(\Sigma'', \phi^* \circ \phi'')$ is an open book for $(M_+, \xi_+)$ and $(\Sigma'',  \phi^*_+ \circ \phi'')$ is an open book for $(-1)$--surgery on $L^*_+$. 
Let $A'' = A \cap \Sigma''$, a four-punctured sphere. (Notice that $L'$ and $L^0_+$ are isotopic in $A''$ to two of the components of $\bdry A''$.  The other two boundary components come from the stabilizations.)  Observe that $\phi^*_+$ is the identity outside of $A''$ and within $A''$ may be represented as in Figure~\ref{fig:lantern}(b).  By the lantern relation, we may express $\phi^*_+$ as $\tau_\gamma \circ \tau_\beta \circ \tau^{-1}_{L'}$ as shown in Figure~\ref{fig:lantern}(c) where $\gamma$ and $\beta$ are the two curves in the middle. (Only one choice of which is which is correct for the order of composition of Dehn twists, but our proof does not require this detail.)  Thus $(\Sigma'', \tau_\gamma \circ \tau_\beta \circ \tau^{-1}_{L'} \circ \phi'')$ is an open book for $(-1)$--surgery on $L^*_+$.  

Let $(M', \xi')$ denote the result of $(+1)$--surgery on $L'$ in $(M,\xi)$, and observe that $(\Sigma, \tau^{-1}_{L'} \circ \phi)$ is an open book for this manifold.  Hence $(\Sigma'', \tau^{-1}_{L'} \circ \phi'')$ is an open book for $(M',\xi')\#(S^1 \times S^2, \xi_{\std})\#(S^1 \times S^2, \xi_{\std})$.  Since $(M',\xi')$ has non-zero Heegaard Floer contact invariant by assumption, so does $(M',\xi')\#(S^1 \times S^2, \xi_{\std})\#(S^1 \times S^2, \xi_{\std})$.  
Then adding positive Dehn twists along $\beta$ and $\gamma$ to the monodromy of the open book $(\Sigma'', \tau^{-1}_{L'} \circ \phi'')$ supporting this manifold results an open book supporting a contact structure that also has non-zero Heegaard Floer contact invariant by \cite{baldwin-comultiplication}.  Therefore, since  $(-1)$--surgery on $L^*_+$ is supported by the open book $(\Sigma'', \tau_\gamma \circ \tau_\beta \circ \tau^{-1}_{L'} \circ \phi'')$, the contact manifold is tight.  Consequentially $L^*_+$ is a non-loose knot in $(M_+,\xi_+)$ and $t_+(L^*) >1$.
\end{proof}  

\begin{figure}
\footnotesize
\centering
\psfrag{L}[][]{$L'$}
\psfrag{Lp}[][]{$L$}
\psfrag{Lpp}[][]{$L^0$}
\psfrag{SLp}[][]{$L^0_+$}
\psfrag{a}[][]{$a$}
\includegraphics[width=5in]{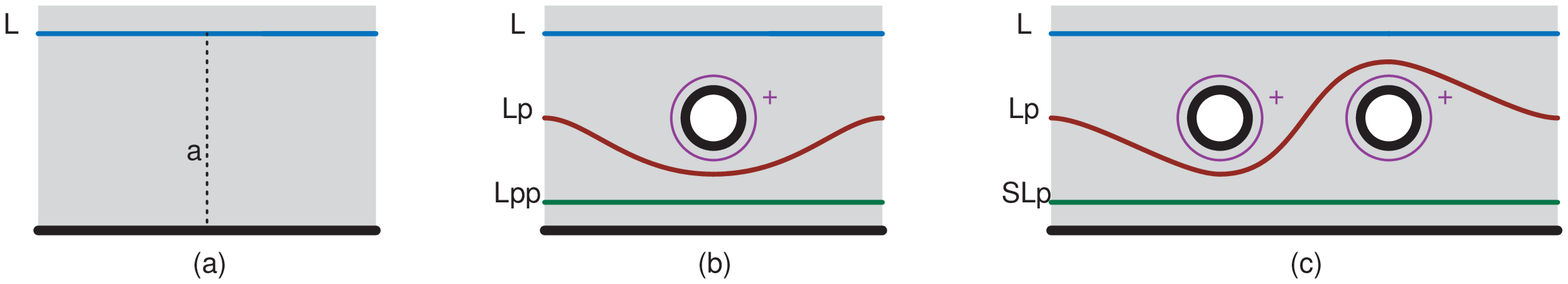}
\caption{}
\label{fig:abstrOB}
\end{figure}

\begin{figure}
\footnotesize
\centering
\psfrag{L}[][]{$L'$}
\psfrag{Lp}[][]{$L$}
\psfrag{Lpp}[][]{$L^0$}
\psfrag{SLp}[][]{$L^0_+$}
\includegraphics[width=5in]{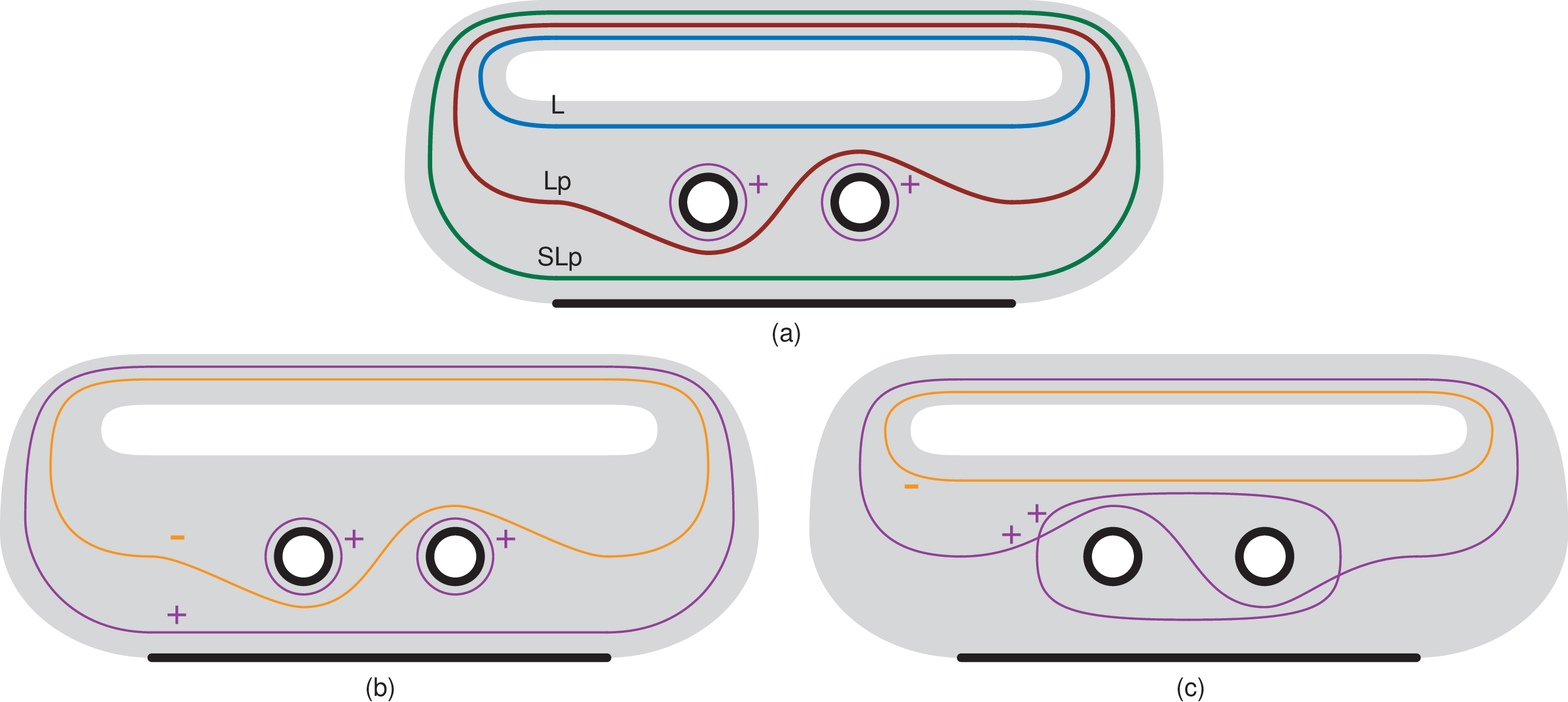}
\caption{}
\label{fig:lantern}
\end{figure}

\begin{example}\label{ex:negtension1}
Using Proposition~\ref{prop:tensionrefinement} one may construct many examples of non-loose Legendrian knots with $t_- = 1$ and $t_+>1$.

Let $(M,\xi)$ be a contact manifold resulting from $(-1)$--surgeries on each component of an oriented Legendrian link $L_0 \cup L_1 \cup \dots \cup L_k$ in $(S^3,\xi_\std)$, and let $L'$ the surgery dual knot to $L_0$.  Then $(+1)$--surgery on $L'$ is the result of a $(-1)$--surgeries on the sublink $L_1 \cup \dots \cup L_k$, a Stein fillable contact manifold which therefore has a non-zero Heegaard Floer contact invariant \cite{osctct}.  Now let $L$ be a positive stabilization of $L'$ in $(M,\xi)$.  Then let $(M_+,\xi_+)$ be the overtwisted manifold resulting from $(+1)$--surgery on $L$ and let $L^*$ be the surgery dual.   According to Proposition~\ref{prop:tensionrefinement}, $t_-(L^*)=1$ and $t_+(L^*)>1$.
\end{example}

\begin{example}\label{ex:unknots}
By Corollary~\ref{cor:nonlooseunknot}, a non-loose unknot has $(\tb,\rot) = (n, \pm(n-1))$ for some positive integer $n$.  
Therefore if $L$ is a non-loose unknot oriented so that $\rot(L) = 1-n$ with $n>0$, $t_-(L)=1$ and $t_+(L) \leq n$.  
If $n\geq 2$, then applying Proposition~\ref{prop:tensionrefinement} as described in Example~\ref{ex:negtension1} to the (local) surgery descriptions of Figure~3 of \cite{plamenevskaya} shows that $t_+(L) \geq 2$. Thus $t_\ort(L) = \{(1,0), (0,1)\}$ if $n=1$ and $t_\ort(L) = \{(0,1)\}$ if $n\geq2$.  
\end{example}

\begin{example}\label{ex:ratlunknots}
In \cite{sinemgeiges}, Geiges-Onaran extend Plamenevskaya's surgery descriptions of non-loose Legendrian unknots in $(S^3, \xi_{-1})$ to non-loose Legendrian rational unknots (i.e.\ knots with solid torus exterior) in overtwisted lens spaces.   They explicitly give surgery diagrams for such rational unknots in the lens spaces $L(p,1)$ for $p\in \N$ and in $L(5,2)$.  Following Example~\ref{ex:negtension1}, we may apply Proposition~\ref{prop:tensionrefinement} to show the non-loose knots described in Figure 6(b)(c), Figure 7(b)(c), and Figure 8(a)(b)(c) of \cite{sinemgeiges} may be oriented to have $t_- = 1$ and $t_+\geq 2$.
\end{example}

\subsection{Order and the LOSS invariant}\label{sec:loss}

Lisca-Ozsv\'ath-Stipsicz-Szab\'o define an invariant $\frL$ which is a $U$--torsion element in $HFK^-(-M,L)$ of an oriented, {\em null-homologous} Legendrian knot $L$ in an overtwisted contact manifold $(M,\xi)$.  As common, we often refer to $\frL$ as the {\em LOSS invariant}.  Lemmas 1.4 and 1.6 of \cite{LOSS} show that $\frL$ enjoys the following properties:
\begin{itemize}
\item If $L$ is loose, then $\frL(L) = 0$. 
\item If $L_-$ is a negative stabilization of $L$, then $\frL(L_-) = \frL(L)$.
\item If $L_+$ is a positive stabilization of $L$, then $\frL(L_+)=U \cdot \frL(L)$.
\end{itemize}

The following definition is suggested by Corollary~1.3 of \cite{LOSS}.
\begin{defn}
Let $L$ be an oriented, null-homologous Legendrian knot in an overtwisted manifold.
Define $o(L)$ to be the {\em order} of the $U$--torsion of $\frL(L)$. That is
\[o(L) = \min \{k\in \N \colon U^k \cdot \frL(L) = 0\}.\]
\end{defn}
\begin{defn}
Let $L$ be an unoriented, null-homologous Legendrian knot in an overtwisted contact structure.  Then the {\em order} of $L$ is $\bar{o}(L) = o(L) + o(-L)$, the sum of the orders of its two orientations (where an arbitrary orientation on $L$ is chosen for the calculation).
\end{defn}

\begin{lemma}\label{lem:stabLOSS}
Let $L$ be an oriented null-homologous Legendrian knot in an overtwisted contact structure.
Given a set of $a$ positive stabilizations and $b$ negative stabilizations such that $L_{+a,-b}$ is loose, then
\[ o(L) \leq a \mbox{ and } o(-L) \leq b.\]
\end{lemma}
\begin{proof}
The result of $b$ negative stabilizations of $L$ is $L_{-b}$, and $\frL(L_{-b}) = \frL(L)$.  Then $L_{+a,-b}$ is the result of $a$ positive stabilizations of $L_{-b}$.  Hence $\frL(L_{+a,-b}) = U^{a} \cdot \frL(L_{-b}) = U^{a} \cdot \frL(L)$.   Since $L_{+a,-b}$ is loose, $\frL(L_{+a,-b})=0$.  Therefore $o(L) \leq a$.  Reversing the orientation on $L$, the other result follows similarly.
\end{proof}

\begin{lemma}\label{lem:tensionLOSS}
If $L$ is a null-homologous Legendrian knot in an overtwisted contact structure then
\[\bar{o}(L)\leq t(L).\]
\end{lemma}
\begin{proof}
Apply Lemma~\ref{lem:stabLOSS} to any partition $a+b$ of $t(L)$ such that $L_{+a,-b}$ is loose.

\end{proof}

\begin{lemma}\label{lem:posorder}
Let $L$ be a non-loose, null-homologous, oriented Legendrian knot.

If $\frL(L) \neq 0$, then all its negative stabilizations are non-loose.  That is, $o(L)>0$ implies $t_-(L) = \infty$.  
Moreover, $L$ must intersect every overtwisted disk negatively. 
\end{lemma}
\begin{proof}
Since negative stabilizations do not change $\frL$, the first statement holds.  
Since positive intersections of an oriented Legendrian knot with an overtwisted disk may be removed by negative stabilizations (Remark~\ref{rmk:stabsign}) and such a knot $L$ with $\frL(L) \neq 0$ cannot be loosened by only negative stabilizations,  there must be a negative intersection of $L$ with any overtwisted disk.
\end{proof}

\begin{example}
Let $\xi_n$ be the overtwisted contact structure on $S^3$ with Hopf invariant $h(\xi_n)=1-2n$.
For each $n \in \N$, Lemma~6.1 \cite{LOSS} gives a family of non-loose Legendrian representatives $L(n)$ of the $(2,2n-1)$--torus knots in $(S^3,\xi_n)$.  Together Proposition~6.2 and Remark~6.3 of \cite{LOSS} show that $\frL(L(n)) \neq 0$; hence $o(L(n))>0$ and, as in Lemma~\ref{lem:posorder}, $t_-(L(n))=\infty$. 

Since stabilizations decrease $\tb$, it follows that for each $n \in \N$ the values of $\tb$ of the non-loose Legendrian representatives of the $(2,2n+1)$--torus knot in $(S^3,\xi_n)$ are not bounded below.

 \end{example}

\begin{prop}\label{prop:mintbLOSS}
If a null-homologous knot type $\calK$ has a lower bound on the Thurston-Bennequin numbers of its non-loose Legendrian representatives in a given overtwisted manifold,  then $\frL(L) = 0$, $o(L)=0$, $\bar{o}(L)=0$, and $t_\pm(L) <\infty$ for each (oriented) Legendrian representative $L$. 
\end{prop}
\begin{proof}
Given any oriented Legendrian representative $L$ of $\calK$, sufficiently many negative stabilizations will produce a representative $L'$ with $\tb(L')$ less that the assumed lower bound.  Hence $L'$ must be loose.  Thus $t_-(L) < \infty$, $\frL(L) = \frL(L') =0$, and so  $o(L)=0$.  A similar argument shows $t_+(L)<\infty$, $o(-L)=0$, and thus $\bar{o}(L) = 0$.
\end{proof}

\begin{cor}\label{cor:unknotorder}
Any Legendrian unknot $L$ in an overtwisted contact structure has $\bar{o}(L) = 0$.
\end{cor}

\begin{proof}
By Theorem~\ref{thm:EF}, any non-loose  Legendrian unknot has $\tb > 0$.
\end{proof}

\begin{remark}\label{rem:sivek}
For any Legendrian knot $L$ in a contact manifold $(M,\xi)$, Sivek defines a monopole Floer invariant $\ell_g(L) \in KHM(-M,L)$ for each integer $g\geq 2$ \cite{sivek} which one may care to compare with the Heegaard Floer LOSS invariant $\frL(L)\in HFK^-(M,L)$.
These invariants $\ell_g$ are all $0$ if either the complement of $L$ is overtwisted or $L=L'_{+1,-1}$ for some other Legendrian knot $L'$ with any orientation.  As such, while these $\ell_g$ may be able to detect non-looseness, they appear to be less suited for bounding tension.  

We caution the reader that Proposition~5.6 of \cite{sivek} contains an error in the penultimate sentence of its proof.  This also impacts Corollary~5.7.  Furthermore the knots constructed in the example following Corollary~5.7 and illustrated in Figure~7 of that article are in fact all loose since their surgery descriptions involve $(+1)$--surgery on a stabilized trefoil.  Consequentially, examples of non-loose Legendrian knots in overtwisted manifolds with non-zero $\ell_g$ have yet to be produced.
\end{remark}

\section{Depth, Tension, and Order for Transverse Knots}\label{sec:trans}
\subsection{Transverse knots and their Legendrian approximations}
Recall the following key facts from section~\ref{sec:transverse}.  Throughout, all knots are oriented.
\begin{itemize}
\item Any two Legendrian approximations of a transverse knot are related by negative stabilizations.
\item The transverse push-off of a Legendrian approximation is equivalent to the original transverse knot.
\item Any invariant of Legendrian knots that is unaltered by negative stabilization is an invariant of transverse knots.
\end{itemize}
Indeed, a transverse knot has only one kind of stabilization.   It may be viewed as the transverse push-off of a positive stabilization of a Legendrian approximation of the original transverse knot.

We define the invariants {\em depth}, {\em tension}, and {\em order} for transverse knots in the obvious manner. Corollary~\ref{cor:finiteloosen}(1) below shows that tension is well-defined.  Since the LOSS invariant $\frL$ of a null-homologous Legendrian knot is unaltered by negative stabilization, the order of a null-homologous transverse knot is defined to be the order of any Legendrian approximation.

\begin{prop}\label{prop:looseLeghasloosepushoff}
Let $L$ be a Legendrian knot in an overtwisted contact structure.
Then $L$ is a Legendrian approximation of a non-loose transverse knot if and only if $t_-(L)=\infty$.
Also $t_-(L) < \infty$ if and only if the transverse push-off of $L$ is loose.
\end{prop}

\begin{remark}
This sharpens Proposition 1.2 of \cite{etnyreovertwisted} by clarifying when a transverse push-off of a non-loose Legendrian knot is loose.
\end{remark}

\begin{proof}
Since transverse push-offs and Legendrian approximations are inverse operations up to negative stabilizations, the two statements of the proposition are equivalent.  We'll prove the second.  First assume $L$ is a Legendrian knot with $t_-(L)<\infty$.  Then $t_-(L)$ negative stabilizations to $L$ produces a loose Legendrian knot $L'$.   Since $L$ and $L'$ have the same transverse push-off, it must be loose.

Now assume the transverse push-off $T$ of $L$ is loose. Then there is an overtwisted disk disjoint from a neighborhood of $T$.  Hence some Legendrian approximation $L'$ of $T$ (taken within this neighborhood of $T$) must be loose.  But since $L$ and $L'$ must have common negative stabilizations, some negative stabilization of $L$ is loose.  Thus $t_-(L)< \infty$.

\end{proof}

\begin{example}
Example~\ref{ex:negtension1} gives a procedure to create examples of Legendrian knots with $t_- = 1$.  Proposition~\ref{prop:looseLeghasloosepushoff} implies each of these knots has a loose transverse push-off.

\end{example}

\begin{remark}
Lemma~\ref{lem:posorder} shows that a non-loose, null-homologous oriented Legendrian knot $L$ with $\frL(L) \neq 0$ (and hence $o(L)>0$) implies $t_-(L)=\infty$.  Proposition~\ref{prop:looseLeghasloosepushoff} then implies the transverse push-off of $L$ is non-loose.  That $\frL(L) \neq 0$ implies $t_-(L)=\infty$ and the transverse push-off of $L$ is non-loose  is effectively the content of the end of the proof of Corollary~7.3 of \cite{LOSS}.
\end{remark}

Proposition~\ref{prop:looseLeghasloosepushoff} has several corollaries.

\begin{cor}\label{cor:finiteloosen}
\quad
\begin{enumerate}
\item There is a finite sequence of stabilizations that loosens a non-loose transverse knot. 
\item Positive stabilizations are required to loosen any Legendrian approximation of a non-loose transverse knot. 
\item If the transverse push-off of a (rationally) null-homologous Legendrian knot is non-loose, then there is no lower bound on the (rational) Thurston-Bennequin number for non-loose Legendrian representatives of this knot type. 
\item A transverse unknot in an overtwisted contact structure is loose.  \cite[Corollary 2.3]{etnyreovertwisted}
 \end{enumerate}
\end{cor}

\begin{proof}
(1) By Theorem~\ref{thm:stabtoloose} a Legendrian approximation of a transverse knot may be loosened by a finite sequence of stabilizations.   Then the transverse push-off of this loosened Legendrian approximation is loose by Proposition~\ref{prop:looseLeghasloosepushoff}.  The positive stabilizations used in loosening the Legendrian approximation correspond to the loosening stabilizations of the original transverse knot.

(2) This is immediate. 

(3) Negative stabilizations must remain non-loose, but each such stabilization decreases the (rational) Thurston-Bennequin number.  Indeed, by Proposition~\ref{prop:mintbLOSS}, if there were a lower bound then $t_- < \infty$ for any Legendrian approximation.

(4) If not, then by (3) there would be no lower bound on $\tb$ for non-loose Legendrian unknot, contrary to Theorem~\ref{thm:EF}.
\end{proof}

\begin{lemma}
If $T$ is a non-loose transverse knot and $L$ is a Legendrian approximation, then $t(T) \leq t(L)$.

\end{lemma}

\begin{proof}
Every stabilization of a transverse knot corresponds to a positive stabilization of its Legendrian approximations.  If  $L$ is loosened by  $p$ positive stabilizations and $n$ negative stabilizations so that $t(L) = p+n$, then $t(T) \leq p \leq t(L)$.
\end{proof}

\begin{lemma}
 Assume a Legendrian approximation $L$ of a non-loose transverse knot $T$ requires at least one positive stabilization and at least one negative stabilization to loosen.   Then $0< t(T) < t(L)$.
\end{lemma}

\begin{proof}
As $T$ is transversally isotopic to the transverse push-off of all negative stabilizations of $L$, it is loosened by only the positive stabilizations.
\end{proof}

\subsection{Bindings of overtwisted open books}
An open book induces a contact structure in which the binding is naturally a transverse link.

\begin{lemma}\label{lem:nonloosebinding}
If a knot $T$ is the connected binding of an open book decomposition that supports an overtwisted contact structure, then $T$ is a non-loose transverse knot in that contact structure.
\end{lemma}

\begin{proof} 
More is actually true. The complement of the binding (connected or not) of an open book for any contact manifold is universally tight.  A proof of this fact appears in the proof of Lemma~3.1 of \cite{EVV}.
\end{proof}

Let $H^+$ and $H^-$ respectively denote the positive and negative Hopf bands in $S^3$.
The open book $(S^3, H^+)$ whose page is $H^+$ supports the tight contact structure on $S^3$.  The open book $(S^3, H^-)$ whose page is $H^-$ supports the overtwisted contact structure $\xi_{-1}$ (with Hopf invariant $-1$) on $S^3$.  
An open book is a {\em positive or negative Hopf stabilization} if it may be obtained by plumbing such a Hopf band onto another open book.

Also recall that if an open book $(M,\Sigma)$ supports the contact structure $\xi$, then $(M, \Sigma \sharp_\alpha H^+)$, the positive Hopf stabilization  obtained by plumbing a positive Hopf band onto the page $\Sigma$ of $(M,\Sigma)$ along an arc $\alpha \subset \Sigma$ also supports $\xi$ \cite{girouxopenbook}. In particular, if $\alpha$ is a boundary-parallel arc in $\Sigma$, $\Sigma \sharp_\alpha H^+$ is the boundary connected sum of the page $\Sigma$ and the band $H^+$.  Indeed, this plumbing may be positioned so that the supported contact structure is identical rather than merely isotopic to the original. Consequentially, the binding of the resulting open book in $\xi$ is the same transverse link $\bdry \Sigma$ with the addition of a transverse unknot encircling as a meridional curve the  component of $\bdry \Sigma$ along which the Hopf band was summed.  A local picture of this is illustrated in Figure~\ref{fig:hopfcircle}.

\begin{figure}
\centering
\psfrag{S}[][]{$\Sigma$}
\psfrag{P}[][]{$\Sigma \sharp H^+$}
\includegraphics[width=3.5in]{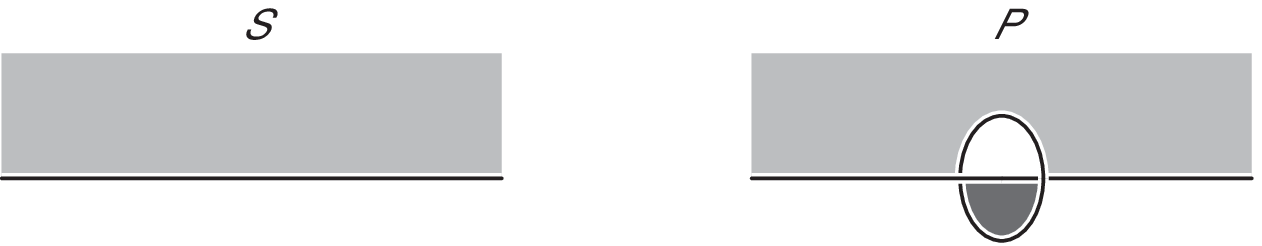} 
\caption{}
\label{fig:hopfcircle}
\end{figure}
 
\begin{lemma}\label{lem:twohopfs}
Let $H^- \sharp H^+$ denote the boundary connected sum of a negative and a positive Hopf band.
The open book $(S^3, H^- \sharp H^+)$ supports a contact structure containing an overtwisted disk $D$ such that $\bdry D$ is a $0$--framed loop in $H^- \sharp H^+$ that is parallel to one component of $\bdry (H^- \sharp H^+)$  and the interior of $D$ intersects each of the other two components just once.
\end{lemma}

\begin{figure}
\centering
\psfrag{S}[][]{$H^- \sharp H^+$}
\psfrag{D}[][]{$\bdry D$}
\includegraphics[width=3.5in]{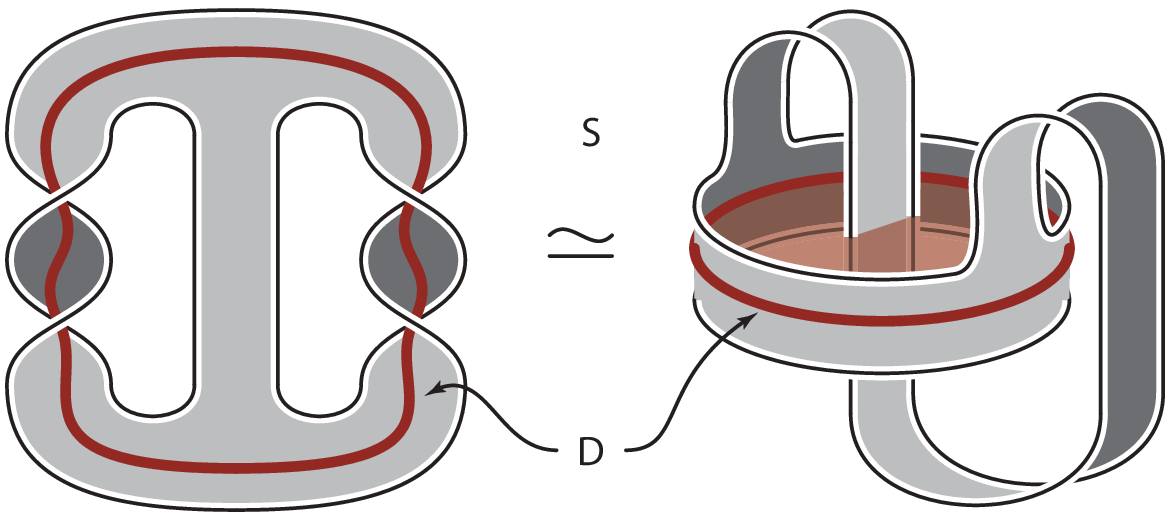} 
\caption{}
\label{fig:twohopfs}
\end{figure}

\begin{proof}
Figure~\ref{fig:twohopfs} shows the surface $H^- \sharp H^+$ with the curve $\bdry D$ on the left and an isotopic presentation of this surface together with the disk $D$ on the right.  Let $\Sigma$ be the union of $H^- \sharp H^+$ with another page of the open book which we may then assume is convex such that the binding is the dividing set \cite{torisu}.
Since $\bdry D$ is not the sole boundary component of a subsurface of $H^- \sharp H^+$ it may be realized as a Legendrian curve for which $\tw(\bdry D, \Sigma)$ equals its page-framing in the supported contact structure \cite{kanda, honda1}.  Since the page framing is the Seifert framing for $\bdry D$ and $\bdry D$ is disjoint from the dividing set, $\tb(\bdry D)=0$.  Hence $D$ is an overtwisted disk with the properties claimed.
\end{proof}

\begin{thm}\label{thm:depth1binding}   
Assume an open book with connected binding is a negative Hopf stabilization.  Then the binding $T$, as the non-loose transverse knot in the overtwisted contact structure the open book supports, has $d(T)=t(T)=1$.
\end{thm}

\begin{proof}
Let $(M,\Sigma)$ be a negatively Hopf stabilized open book for $M$ with page $\Sigma$ and connected binding $T=\bdry \Sigma$ that supports the contact structure $\xi$. By Lemma~\ref{lem:nonloosebinding}, $T$ is a non-loose transverse knot and thus $d(T)>0$.   By assumption $(M,\Sigma) = (M, \Sigma_0 \sharp_\beta H^-)$, the plumbing of the negative Hopf band onto the open book $(M,\Sigma_0)$ along an arc $\beta \subset \Sigma_0$.   Plumb the positive Hopf band onto $(M,\Sigma)$ along a boundary parallel arc $\alpha \subset \Sigma$ to form the open book $(M, \Sigma \sharp_\alpha H^+)$; that is, take the boundary connected sum of $\Sigma$ with $H^+$.  Since this is a positive stabilization of the original open book $(M,\Sigma)$, it supports the same contact structure $\xi$.  As discussed above, the binding of $(M, \Sigma \sharp_\alpha H^+)$ is the link $T \cup \mu$ consisting of the transverse knot $T$ and a transverse meridian (as pictured in Figure~\ref{fig:hopfcircle}).

We now find an overtwisted disk whose interior is intersected twice by $T \cup \mu$, once by each.  This will exhibit a depth $1$ overtwisted disk for $T$.  
The boundary connect sum of $H^+$ with $\Sigma$ may be done at a point on $\bdry H^-$ disjoint form $\Sigma_0$.  
Then we may view $(M, \Sigma \sharp_\alpha H^+)$ as the plumbing of $(S^3, H^- \sharp H^+)$ onto $(M,\Sigma_0)$ along a rectangle in $H^- \sharp H^+$ disjoint from $H^+$. With this plumbing, the overtwisted disk $D$ of Lemma~\ref{lem:twohopfs} carries over to an overtwisted disk in $(M, \Sigma \sharp_\alpha H^+)$ where its boundary is a $0$--framed curve in the page and its interior intersects each $T$ and $\mu$ just once.  Hence $d(T)=1$.
\end{proof}

\begin{remark}
Depth (and tension) can be defined for links.
The proof of Theorem~\ref{thm:depth1binding} easily extends to negatively Hopf stabilized open books with bindings of multiple components, showing that there is an overtwisted disk intersected just once by the binding.  As such, Theorem~\ref{thm:depth1binding} may be viewed as offering an obstruction to an open book being a negative Hopf stabilization.
\end{remark}

\begin{example}
Let $T_1 \cup T_2 = \bdry H^-$ be the transverse link arising as the binding of the open book $(S^3, H^-)$ in the supported overtwisted contact structure $\xi_{-1}$.  By Lemma~\ref{lem:nonloosebinding} the link is non-loose.  However since $T_1$ and $T_2$ are both unknots, individually they are loose transverse knots according to Corollary~\ref{cor:finiteloosen}(4).  This implies that $T_1$ intersects every overtwisted disk that is disjoint from $T_2$ and $T_2$ intersects every overtwisted disk that is disjoint from $T_1$.  Theorem~\ref{thm:depth1binding}, extended for disconnected binding, implies $d(T_1 \cup T_2)=1$.
\end{example} 

\begin{remark}
Ito-Kawamuro use their theory of open book foliations to view overtwistedness of a contact structure through a ``transverse overtwisted disk'' with respect to an open book supporting the contact structure \cite{itokeiko1, itokeiko2}.  Such disks will give an upper bound on the depth of the binding of the open book; see Proposition 4.2~\cite{itokeiko1}.
\end{remark}



\section{Problems and Questions} \label{sec:prob}

\begin{problem}
Develop constructions of knots of large depth or tension.
\end{problem}

Topological operations on knots and links often have Legendrian analogues.  It is natural to ask how our invariants of depth and tension behave under these operations.  Two fundamental operations are Legendrian Whitehead doubles \cite{fuchs,ngtraynor, etnyrelegandtrans} and Legendrian cables \cite{EHcable,dinggeigescable}. While some sources only explicitly define these operations for Legendrian knots in $(\R^3,\xi_\std)$, their definitions other contact manifolds are straightforward.

\begin{figure}[h!]
\footnotesize
\centering
\psfrag{p}[][]{$p$}
\psfrag{q}[][]{$q$}
\psfrag{n}[][]{$n$}
\psfrag{C}[][]{$C_{p,q}$}
\psfrag{c}[][]{$C_{p,-q}$}
\psfrag{W}[][]{$W_n$}
\psfrag{w}[][]{$W_{-n}$}
\includegraphics[width=4in]{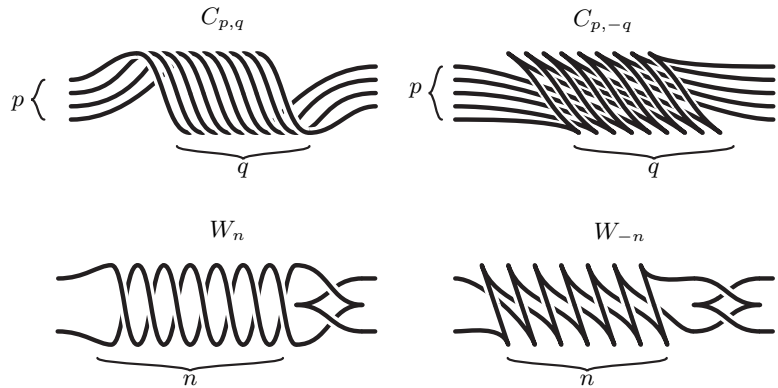}
\caption{Legendrian cables and twisted Whitehead doubles for the $x$--axis in the front projection, modulo translation.}
\label{fig:cablewhiteheadLeg}
\end{figure}

For coprime integers $p,q$ with $p>0$, let $C_{p,q}(L)$ denote the Legendrian $(p,q)$--cable of the Legendrian knot $L$.  For an integer $n$, let $W_n(L)$ denote the $n$--twisted Legendrian Whitehead double of the Legendrian knot $L$.  These satellites, taken in a standard contact solid torus neighborhood of $L$, are illustrated in Figure~\ref{fig:cablewhiteheadLeg}. 
\begin{question}
For a non-loose Legendrian knot $L$, how are the depth, tension, and order of $L$ and $C_{p,q}(L)$ related?  In particular, given coprime positive integers $p$ and $q$, does $d(C_{p,q}(L)) = p\cdot d(L)$?  When is $C_{p,q}(L)$ loose?
\end{question}

\begin{question}
For a non-loose Legendrian knot $L$, how are the depth, tension, and order of $L$ and $W_n(L)$ related?  In particular, if $n\geq0$ does $d(W_{n}(L)) = 2\cdot d(L)$?  When is $W_n(L)$ loose?
\end{question}

\begin{problem}
Construct explicit examples of Legendrian knots $L$ with $d(L)=2$.  
\end{problem}
Of course Theorems~\ref{thm:bigdepth} and \ref{thm:depth2} as well as Remark~\ref{rem:depth2constructions} should be kept in mind.  The previous two questions suggest investigating the cables $C_{2,q}(L^*)$ and Whitehead doubles $W_n(L^*)$ of the $(+1)$--surgery duals to a stabilized knot.

\begin{question}
Do the $(+1)$--surgery duals to the Legendrian knots in $(S^3,\xi_\std)$ having a local configuration as in Figure~1 of \cite{liscastipsicz} all have depth at most $2$?  In particular, do the surgery duals to $(+1)$--surgery on negative torus knots with maximal Thurston-Bennequin number all have depth $2$?
\end{question}
Lisca-Stipcisz show that $(+1)$--contact surgery on these knots produce an overtwisted contact structure by constructing a once-punctured torus with Legendrian boundary in the surgered manifold that violates the Bennequin-type inequality for tight contact structures.  Being duals to knots in a tight manifold, they are necessarily non-loose.  At the very least one should be able to determine an upper bound on depth for all such knots (such as in the discussion preceding Theorem~\ref{thm:depth2}).

\medskip

\begin{problem}
Study the discrepancy between depth and tension.
\begin{enumerate}
\item  How big can $d-t$ be among Legendrian knots? 
\item Does every knot type have an upper bound on $d-t$ for its Legendrian representatives in each overtwisted contact structure?
\item Is there a null-homologous Legendrian knot with $d-t>0$?
\end{enumerate}
While the proof of Theorem~\ref{thm:tensionlessthandepth} only exhibits {\em rationally} null-homologous Legendrian knots with $d-t>0$, there should be null-homologous examples.
Indeed, we expect that every overtwisted manifold contains non-loose Legendrian knots for which this difference can be arbitrarily large, though not within a single knot type.
\end{problem}

\begin{question} 
Can both the positive and negative tension of a non-loose Legendrian knot be large?
\begin{enumerate}
\item Is there a non-loose Legendrian knot $L$ for which $t_+(L)$ and $t_-(L)$ are both bigger than $t(L)$?  That is, is there a non-loose Legendrian which requires both a positive and a negative stabilization to realize its tension?  
\item Are there non-loose Legendrian knots with both $t_+ = \infty$ and $t_-=\infty$?  A Legendrian Whitehead double of a knot with $t_-=\infty$ seems like a likely candidate. 
\end{enumerate}
\end{question}

\begin{problem}
Determine the overtwisted contact structures and knot types for which there is a lower bound on the Thurston-Bennequin number of the non-loose Legendrian representatives.  

If such a knot type is rationally null-homologous, then Corollary~\ref{cor:finiteloosen}(3) with Proposition~\ref{prop:looseLeghasloosepushoff} implies that every transverse representative is loose and $t_- < \infty$ for each Legendrian representative.  If it is actually null-homologous, then by Proposition~\ref{prop:mintbLOSS} $\frL=0$ and $o=0$ as well.
\end{problem}

\begin{question}
If an open book with connected binding $K$ supports an overtwisted contact structure, must $t(K)=1$?

We expect the answer to be no.  For comparison, Theorem~\ref{thm:depth1binding} says that $t(K)=d(K)=1$ whenever the open book admits a negative stabilization. 
\end{question} 

\begin{problem}
Extend the LOSS invariant to rationally null-homologous knots, and study the order of rationally null-homologous non-loose knots.
\end{problem}

\begin{question}
Are there non-loose transverse knots for which every Legendrian approximation has $\frL=0$?
\end{question}

\begin{problem}
Find a non-loose Legendrian knot with non-zero $\ell_g$; see Remark~\ref{rem:sivek}. The non-loose unknots in $(S^3, \xi_{-1})$ are reasonable candidates.
\end{problem}

Recall that we have been working under the assumption that our overtwisted manifolds are closed.  The definitions of depth and tension clearly extend to overtwisted manifolds with boundary in which some overtwisted disk is contained in the interior of the manifold.

 However, if every overtwisted disk in an overtwisted contact manifold with boundary were properly embedded, then stabilizations could not loosen a non-loose knot.  Indeed 
Vela-Vick pointed out that Theorem~\ref{thm:stabtoloose} fails if $(V,\xi)$ is the contact solid torus $V= \{(r,\theta,z) \vert r \leq \pi\}/(z \mapsto z+1)$ with $\xi = \ker(\cos r \, dz + r \sin r \, d\theta)$.  Any Legendrian knot in $V$ not contained in a ball is necessarily non-loose.

\begin{problem}
Characterize overtwisted manifolds with boundary in which every overtwisted disk is properly embedded. 
\end{problem}

\bibliographystyle{amsalpha}
\bibliography{nonloose}

\end{document}